\documentclass[12pt, twoside]{article}

\usepackage[style=alphabetic,natbib=true,doi=false,isbn=false,backend=biber]{biblatex}
\usepackage[leqno]{amsmath}
\usepackage{amsthm}
\usepackage{amssymb,amscd,latexsym,stmaryrd}
\usepackage[shortlabels]{enumitem}
\setlist[enumerate]{nosep}
\usepackage[colorlinks,pdfpagelabels,pdfstartview = FitH,bookmarksopen = true,bookmarksnumbered = true,linkcolor = black, urlcolor = black,plainpages = false,hypertexnames = false,citecolor = black]{hyperref}
\usepackage{bm}
\usepackage{fonttable}           
\usepackage[english]{babel}
\usepackage{tocloft}
\usepackage{aurical}
\usepackage{csquotes}
\addto{\captionsenglish}{}
\usepackage[T1]{fontenc}
\usepackage[bbgreekl]{mathbbol}
\usepackage[md]{titlesec}
\titleformat{\subsubsection}{\normalfont\itshape}{\thesubsubsection.}{1mm}{}
\titleformat{\section}{\normalfont\centering}{\thesection.}{1mm}{}
\titleformat{\subsection}[runin]{\normalfont}{\thesubsection.}{1mm}{}
\usepackage[all]{xy}

\newcommand{\A}{{\mathbb{A}}}

\newcommand{\F}{{\mathbb{F}}}

\newcommand{\N}{{\mathbb{N}}}

\newcommand{\Q}{{\mathbb{Q}}}

\newcommand{\T}{\ensuremath{\mathcal{T}}}

\newcommand{\Z}{{\mathbb{Z}}}

\newcommand{\rsim}{\mathbin{\text{\rotatebox[origin=c]{90}{$\sim$}}}}

\newcommand{\sHom}{\mathop{\mathcal{H}\! \mathit{om}}\nolimits}

\newcommand{\la}{\langle}
\newcommand{\ra}{\rangle}

\newcommand{\FF}{\ensuremath{\mathcal{F}}}
\newcommand{\GG}{\ensuremath{\mathcal{G}}}

\newcommand{\ev}{\mathrm{ev}}

\newcommand{\OO}{\ensuremath{\mathcal{O}}}

\newcommand{\id}{\mathrm{id}}

\newcommand{\pr}{\mathrm{pr}}

\renewcommand{\mod}{\;\mathrm{mod}\;}
\renewcommand{\epsilon}{\varepsilon}

\newcommand{\ord}{\mathrm{ord}}

\newcommand{\res}{\mathrm{res}}

\newcommand{\spec}{\mathrm{Spec}\,}

\newcommand{\Hom}{\mathrm{Hom}}

\DeclareMathOperator{\lcm}{lcm}

\DeclareMathOperator{\gh}{gh}

\DeclareMathOperator{\can}{can}

\DeclareMathOperator{\im}{im}

\DeclareMathOperator{\proj}{Proj}

\newcommand{\Af}{\ensuremath{\mathbb{A}}}

\newtheorem{theorem}{Theorem}[section]
\newtheorem{lemma}[theorem]{Lemma}
\newtheorem{theodef}[theorem]{Theorem/Definition}
\newtheorem{prop}[theorem]{Proposition}
\newtheorem{cor}[theorem]{Corollary}

\newtheorem{theorems}{Theorem}[subsection]
\newtheorem{lemmas}[theorems]{Lemma}
\newtheorem{cors}[theorems]{Corollary}
\newtheorem{props}[theorems]{Proposition}

\newtheorem{lemmaA}{Lemma}

\newtheorem{lemmaAIntro}{Lemma}

\newtheorem{theoremB}{Theorem}

\newtheorem{theoremBIntro}{Theorem}

\newtheorem{theoremC}{Theorem}

\newtheorem{theoremCIntro}{Theorem}

\theoremstyle{definition}
\newtheorem{definition}[theorem]{Definition}

\theoremstyle{remark}
\newtheorem{remark}[theorem]{Remark}

\newtheorem{example}[theorem]{Example}

\newtheorem{exams}[theorems]{Example}

\titlespacing*{\section}{0pt}{0mm}{0mm}
\let\oldthepage\thepage
\newcommand*\thefootpage{\footnotesize{\oldthepage}}
\usepackage{fancyhdr}
\fancypagestyle{plain}{%
	
	\fancyhf{}\fancyfoot[C]{}\fancyhead[RE,RO]{\thefootpage}\fancyhead[CE]{\footnotesize{\textsc{maria lünnemann}}}
	\fancyhead[CO]{\footnotesize{\textsc{a geometric approach to the relative de rham-witt complex}}}%
}

\usepackage{theoremref}
\newcommand*{\defeq}{\mathrel{\vcenter{\baselineskip0.5ex \lineskiplimit0pt
			\hbox{\scriptsize.}\hbox{\scriptsize.}}}%
	=}
\newcommand{\p}{\ensuremath{\mathfrak{p}}}
\newcommand{\q}{\ensuremath{\mathfrak{q}}}
\usepackage{tikz-cd}
\newenvironment{mycenter}[1][\topsep]
{\setlength{\topsep}{#1}\par\kern\topsep\centering}
	{\par\kern\topsep}

\usepackage{bbm}
\newcommand{\dd}{\ensuremath{\mathbbm{d}}}
\begin{document}
	\setlength{\abovedisplayskip}{5pt}
	\setlength{\belowdisplayskip}{5pt}
\title{\normalsize\textsc{\textbf{A GEOMETRIC APPROACH \\TO THE RELATIVE DE RHAM-WITT COMPLEX\\
IN THE SMOOTH, $\Z$-TORSION FREE CASE}\vspace{-5mm}}}
\author{\normalsize\textsc{maria lünnemann}\footnote{funded by the Deutsche Forschungsgemeinschaft (DFG, German Research Foundation) under Germany's Excellence Strategy EXC 2044--390685587, Mathematics M\"unster: Dynamics--Geometry--Structure and by the CRC 1442 Geo\-me\-try: Deformations and Rigidity}}
\date{}
 \maketitle\vspace{-13mm}
\thispagestyle{empty}
\begin{center}
	\begin{minipage}{30em}
		\small \textsc{Abstract.} Let $X$ be a smooth scheme over a finitely generated flat $\Z$-, $\Z_{(p)}$- or $\Z_p$-algebra $R$. Evaluated at finite truncation sets $S$, the relative de Rham-Witt complex $W_S\Omega_{X/R}^{\bullet}$ is a quotient of the de Rham complex $\Omega^{\bullet}_{W_S(X)/W_S(R)}$, which can be computed affine locally via explicit, but complicated relations. In this paper we prove that $W_S\Omega_{X/R}^{\bullet}$ is the torsionless quotient of the usual de Rham complex $\Omega^{\bullet}_{W_S(X)/W_S(R)}$ on the singular scheme $W_S(X)$. This result was suggested by comparison with a similar modification of the de Rham complex in the theory of singular varieties.
	\end{minipage}
\end{center}
\section*{\scshape{Introduction}}\addcontentsline{toc}{section}{\protect\numberline{}Introduction}
By means of \citeauthor{Hes14}'s construction of the absolute de Rham-Witt complex \cite{Hes14} \citeauthor{Cha13} defined the relative de Rham-Witt complex by linea\-rizing the differential with respect to a base ring \cite{Cha13}. 
Moreover, he used \citeauthor{LZ04}'s computations for polynomial algebras \cite{LZ04} as well as Borger's and van der Kallen's \'etale base change theorem for Witt vector rings \cites{Bor15a,vdK86} to show that the relative de Rham-Witt complex $W_S\Omega^{\bullet}_{X/R_0}$ of a smooth scheme $X$ over a flat $\Z$-algebra $R_0$ is $\Z$-torsion free.
Using this result a direct and much simpler description of the relative de Rham-Witt complex was given by \citeauthor{CD15} \cite{CD15}.\\
Recently, Bhatt, Lurie and Mathew \cite{BLM20} gave a new construction of a ``saturated de Rham-Witt complex'' over a perfect field of characteristic $p$. As we mostly focus on $\Z$-torsion free base rings, we do not discuss their approach.\\

The relative de Rham-Witt complex is not an easy object to define -- so what is its conceptual meaning?
Philosophically, the sheafified de Rham-Witt complex $W_S\Omega_{ X/R_0}^{\bullet}$ should be interpreted as a modification of the naive de Rham complex $\Omega_{ W_{S}(X)/W_S(R_0)}^{\bullet}$, which is required because the scheme $W_S(X)$ is singular over $\spec W_S(R_0)$. Our \hyperref[TheoremMAIN]{main result} can be interpreted as a realization of this philosophy (\thref{theoremmain}).\\
Inspired by known modifications of the de Rham complex of a singular variety, we investigated modifications of differential forms on the singular scheme $W_S(X)$ over $W_S(R_0)$: 
First we will show that the de Rham-Witt complex is not reflexive. The most obvious objects to analyse afterwards are the torsion free diffe\-rential forms. For technical reasons, the torsionless differential forms are much more convenient in our setting. It turns out that the torsionless differential forms form a sheaf for the \'etale topology of $W_S(X)$, which can be equipped with the extra structures of the relative de Rham-Witt complex.
We now explain the required notions and state the precise result.\footnote{In \cite[chapter 4]{Lue22} there is a more detailed analysis of torsion free and reflexive differential forms as well as differential forms in the $h$-topology in connection with the relative de Rham-Witt complex.}

Without further notice we assume all rings $R$ to be commutative unital. Every $R$-module $M$ is naturally equipped with an evaluation map
$$\epsilon_{M}\colon M\to M^{\vee\vee}\defeq \Hom_{R}(\Hom_{R}(M,R),R)$$
to its bidual module. The coimage $T_R(M)\defeq M/\ker\epsilon_M$ is called the \textit{torsionless quotient} of $M$. 
The notion can be extended to sheaves $\FF$ of $\OO_X$-modules: The torsionless quotient  $\T(\FF)$ is given by the sheaf $\FF/\ker \epsilon_{\FF}$, where $\epsilon_{\FF}$ denotes the ``evaluation map''
$$\epsilon_{\FF}\colon \FF\to \FF^{\vee\vee}\defeq \sHom_{\OO_X}(\sHom_{\OO_X}(\FF,\OO_X),\OO_X).$$
For coherent $\OO_X$-modules on locally noetherian schemes the torsionless quotient is again coherent and commutes with flat base change.\\
We call a morphism $X\to Y$ of flat $\Z$-schemes with $X$ locally noetherian \textit{generically smooth} if it is locally of finite type and the base change $ X_{\Q}\to  Y_{\Q}$ is smooth. A ring homomorphism is called generically smooth if the associated morphism of schemes is generically smooth.

\begin{lemmaAIntro}
	Given a generically smooth morphism $f\colon X\to Y$ of schemes, the torsionless sheaf $\T(\Omega^{\bullet}_{X/Y})$ is an $f^{-1}\OO_Y$-differential graded algebra (or an $f^{-1}\OO_Y$-dga for short).
\end{lemmaAIntro}

Moreover, the functor $(U\to X)\mapsto \T(\Omega^q_{U/Y})(U)$
defines a coherent sheaf of $\OO$-modules for the \'etale topology of $X$ (see \thref{etalesheaf}).\\
Coming back to morphisms of Witt vector rings with respect to finite truncation sets $S$, we set $T_{W_{S}}(\Omega_{R/R_0}^{\bullet})\defeq T_{W_S(R)}(\Omega_{W_S(R)/W_S(R_0)}^{\bullet})$ for a ring homomorphism $R_0\to R$ (and ana\-logously for $\T$).\\
In order to obtain a Witt complex (see \cite[Def. 4.1]{Hes14}) we furthermore need the following finiteness condition:
We call a morphism $X\to Y$ of schemes $W_S\hspace{0.2mm}$\textit{-finite} (or $W$\textit{-finite}) if $W_S(X)$ is locally noetherian and the induced morphism $W_S(X)\to W_S(Y)$ is locally of finite type for a (or for every) finite truncation set $S$ -- and analogously for ring homomorphisms.\\
This condition is satisfied in the commonly studied situations: If $R_0$ is a finitely generated $\Z$-, $\Z_{(p)}$- or $\Z_p\hspace{0.2mm}$-algebra and $f\colon X\to \spec R_0$ locally of finite type, then $f$ is $W$-finite.

\begin{theoremBIntro}\thlabel{TheoremB}
	For a generically smooth, $W$-finite ring homomorphism $R_0\to  R$ the functor
	\begin{equation*}
		(\text{finite truncation sets})\to (W(R_0)\text{-dga's}),~
		S\mapsto T_{W_{S}}(\Omega_{R/R_0}^{\bullet})
	\end{equation*}
	defines a Witt complex over $R$ with $W(R_0)$-linear differential.\\
	For an  \'etale homomorphism $R\to R'$ and a finite truncation set $S$ the induced map
	\begin{equation*}
		T_{W_S}(\Omega^q_{R/R_0})\otimes_{W_S(R)}W_S(R') \to T_{W_S}(\Omega^q_{R'/R_0})
	\end{equation*}
	is an isomorphism.
\end{theoremBIntro}

Note that every Witt complex is uniquely determined -- up to isomorphism -- by its values on finite truncation sets $S$. For infinite truncation sets it is obtained as the projective limit over the directed poset of finite truncation subsets $S'\subseteq S$.\\
For the proof of \thref{TheoremB} we establish the structure of a Witt complex on the torsionless quotient via an embedding $T_{W_{S}}(\Omega_{R/R_0}^{\bullet})\hookrightarrow (\Omega_{R_{\Q}/R_{0,\Q}}^{\bullet})^S$ induced by the ghost maps on $R_0$ and on $R$. The examination of the compatibility relations is straightforward.
The last part of the theorem follows from the fact that the functor $W_S\colon (\text{Rings})\to (\text{Rings})$ preserves \'etaleness for finite truncation sets $S$ \cites{Bor15a,vdK86}.\\
Finally, by using \citeauthor{LZ04}'s \cite{LZ04} analysis with respect to polynomial algebras we show that the de Rham-Witt complex $W_S\Omega_{X/R_0}^{\bullet}$ is torsionless. 
This induces a morphism of Witt complexes $\T_{W_S}(\Omega_{X/R_0}^{\bullet})\to W_S\Omega_{X/R_0}^{\bullet}$,
which turns out to be an isomorphism:
\begin{theoremCIntro}\thlabel{theoremmain}\label{TheoremMAIN}
	Let $R_0$ be a flat $\Z$-algebra, $f\colon X\to \spec R_0$ a smooth morphism of schemes and $S$ a finite truncation set. 
	If $f$ is $W_S\hspace{0.2mm}$-finite, then $\T_{W_S}(\Omega^{\bullet}_{X/R_0})$ is a $W_S(f)^{-1}\OO_{\spec W_S( R_0)}$-dga. In degree $q\geq 0$ it is the unique coherent sheaf of $W_S(\OO)$-modules $\T_{W_S}(\Omega^{q}_{X/R_0})$ for the \'etale topology of $W_S(X)$ such that 
	$$\T_{W_S}(\Omega^{q}_{X/R_0})(W_S(U))=W_S\Omega^q_{R/R_0}$$
	holds for every \'etale map $U=\spec R\to  X$.
	If moreover $f$ is $W$-finite, then the functor
	\begin{equation*}
		(\text{finite truncation sets})\to (W(R_0)\text{-dga's}),~
		S\mapsto  T_{W_{S}}(\Omega_{R/R_0}^{\bullet})
	\end{equation*}
	defines the relative de Rham-Witt complex over $R$ with $W(R_0)$-linear differential.
\end{theoremCIntro}
In the appendix we investigate the geometry of Witt schemes. For an $S$-\textit{torsion free} scheme $X$, i.e.  multiplication by elements in $S$ is injective on $\OO_X$,  we show that the ghost map  $\gh_S\colon \coprod_{n\in S} X\to W_S(X)$ is obtained by a finite sequence of explicit blowing ups (see \thref{blowupS}). If moreover $X$ is normal, then the ghost map $\gh_S$ is simply the normalization of $W_S(X)$ in its total ring of fractions. In appendix \ref{AppendixB} we prove a generalization of Dwork's Lemma\footnote{suggested by \citeauthor{Chat12}} for the relative de Rham-Witt complex in the language of Cuntz and Deninger \cite[sec. 3]{CD15} (see \thref{DworkdeRham}).

\subsection*{\textbf{Acknowledgements.}}\addcontentsline{toc}{section}{\protect\numberline{}Acknowledgements} This paper is an abridged version of my dissertation ``A geometric approach of the relative de Rham-Witt complex'' \cite{Lue22}.
I am very grateful to my advisor Christopher Deninger for his support, his positive nature and professional feedback.
\tableofcontents
\section[Preliminaries]{\scshape{Preliminaries}}
In this section we quickly recall some basic facts on Witt rings, Witt schemes and the relative de Rham-Witt complex following Hesselholt \cite{Hes14}, Chatzistamatiou \cites{Chat12}{Cha13} and Borger \cites{Bor15a}{Bor15}.

A subset $S\subseteq \N$ is called \textit{truncation set} if it is stable under division, i.e. $k\in \N, n\in S$ with $k|n$ implies $k\in S$.
We view a non-empty truncation set $S$ as a partially ordered set. For integers $m,n\in \N$ we use the following notation:

\begin{tabular}{ll}
	$m\vert n$  &$m$ divides $n$  \\
	$m\Vert n$ & $m$ divides $n$ and $m\neq n$  \\
	$[m,n]$& least common multiple of $m$ and $n$  \\
	$(m,n)$ & greatest common divisor of $m$ and $n$  \\
	$S/n$& the set of all integers $k\in S$ such that $k n\in S$  \\
	$S(n)$& the set of all integers $k\in S$ with $n\nmid k$
\end{tabular}
\subsection[Witt vector rings and Witt schemes]{\textbf{Witt vector rings and Witt schemes.}}\label{section Witt vector rings and Witt schemes}
For a ring $R$ and a truncation set $S$ the \textit{ring of Witt vectors} $W_S(R)$ is given by $R^S\defeq\prod_{n\in S} R$ as a set such that the \textit{ghost map}
$$\gh_S=(\gh_n)_{n\in S}\colon W_S(R)\to R^S,~ \gh_n((r_k)_{k\in S})=\sum_{k|n}k\cdot r_k^{n/k}$$
is a functorial ring homomorphism.\\
For all $n\in\N$ there are a functorial ring homomorphism $F_n\colon W_S(R)\to W_{S/n}(R)$, called \textit{Frobenius} which is defined by $\gh_k\circ F_n=\gh_{kn}$, and a functorial homomorphism of $W_S(R)$-modules $V_n\colon W_{S/n}(R)\to W_S(R), (r_k)_{k\in S/n}\mapsto (\delta_{n|k}r_{k/n})_{k\in S}$, called \textit{Verschiebung}. Here we have $\delta_{n|k}=1$ if $n|k$ and $\delta_{n|k}=0$ otherwise. These maps satisfy the relations
$F_nV_n=n$, $F_nV_m=V_mV_n$ if $(m,n)=1$ as well as the projection formula
$$V_n(F_n(r)\cdot s)=r\cdot V_n(s),~ r\in W_S(R), s\in W_{S/n}(R).$$
Moreover, there is the multiplicative \textit{Teichm\"uller} map $$[.]\defeq [.]_S\colon R\to W_S(R),~{r\mapsto (r,0,0,\dots)}.$$
If $R$ is equipped with Frobenius lifts, then Dwork's Lemma characterizes the image of the ghost map by a system of congruences:
\begin{lemma}[{\cite[Lemma 1.1]{Hes14}}]\label{Dwork'sLemma}
	\thlabel{Dwork} 
	Given a truncation set $S$, let $R$ be a ring equipped with Frobenius lifts $\phi_p\colon R\to R$ for every prime number $p\in S$, i.e. homomorphisms such that $\phi_p(r)- r^p\in pR$ for all $r\in R$.
	For every $r=(r_n)_{n\in S}\in R^S$ the following are equivalent:
	\begin{enumerate}[(i)]
		\item the element $r$ lies in the image of the ghost map $\gh_S\colon W_S(R)\to R^S$;
		\item for every prime number $p\in S$ and all $n\in p\cdot S/p$ we have $\phi_p(r_{n/p})- r_{n}\in p^{\ord_p(n)}R$.
	\end{enumerate}
\end{lemma}
\begin{lemma}[{\cite[Lemmas  1.0.7 \& 1.0.9]{Chat12}}]\thlabel{Wittlocali}
	Let $S$ be a finite truncation set and $R$ a ring.
	\begin{enumerate}[(i)]
		\item  For a multiplicative subset $\mathcal{S}\subseteq R$ we denote $[\mathcal{S}\defeq\{[s]| s\in \mathcal{S}\}\subseteq W_S(R)$. The map $[\mathcal{S}]^{-1}W_S(R) \to W_S(\mathcal{S}^{-1}R)$ is a ring isomorphism.
		\item  For a multiplicative subset $\mathcal{S}\subseteq \Z$ the map $W_S(R)\otimes_{\Z} \mathcal{S}^{-1} \Z\to W_S(\mathcal{S}^{-1}R)$ is a ring isomorphism.
	\end{enumerate}
\end{lemma}
\begin{definition}
	Let $S$ be a truncation set.
	We call a ring $R$ $S$\textit{-torsion free}, if multiplication by elements in $S$ is injective on $R$.
\end{definition}
Moreover, we define the ring $\Q_S\defeq \Z[1/p\mid p\in S]$ of $S$\textit{-rational numbers}.
Given a ring $R$ and an $R$-module $M$, we set $M_{\Q_S}\defeq M\otimes_{\Z}\Q_S$.
\begin{lemma}[{\cites[Lemma 1.3]{CD15}[Remark 2.5]{Rab14}}]
	Let $S$ be any truncation set.\thlabel{tensorqsghost}
	\begin{enumerate}[(i)]
		\item For an $S$-torsion free ring $R$ the ghost map $\gh_S\colon W_S(R)\to R^S$ is injective. If moreover $S$ is finite, then $\gh_S$ induces a ring isomorphism $W_S(R)_{\Q_S}\to R^S_{\Q_S}$.
		\item For a $\Q_S$-algebra $R$ the ghost map $\gh_S\colon W_S(R)\to R^S$ is a ring isomorphism.\thlabel{ghostisom}
	\end{enumerate} 
\end{lemma}
\begin{theorem}[{\cites[Thm. 1.25]{Hes14}[Thm. B]{Bor15a}[Cor. 15.4]{Bor15}[Thm. 2.4]{vdK86}}]\thlabel{BorKal}
	Let $f\colon A\to B$ be an \'etale ring homomorphism, $S$ a finite truncation set and $n\in\N$ an integer.
	The induced ring homomorphism $W_S(f)\colon W_S(A)\to W_S(B)$ is \'etale and the ring homomorphism
	\begin{equation*}
		W_{S/n}(A)	\otimes_{F_n,W_S(A)}W_S(B)\to W_{S/n}(B),~		a\otimes b\mapsto W_{S/n}(f)(a)\cdot F_n(b)
	\end{equation*}
	is an isomorphism.
\end{theorem}
\begin{cor}[{\cite[Lemma 1.0.15, Cor. 1.0.16]{Chat12}}]
	Let $f\colon A\to B$ be an \'etale ring homomorphism and $S$ a finite truncation set.
	\begin{enumerate}[(i)]
		\item For any truncation subset $S'\subseteq S$ the ring homomorphism
		\begin{equation*}
			W_{S'}(A)	\otimes_{W_S(A)}W_S(B)\to W_{S'}(B),~
			a\otimes b\mapsto W_{S'}(f)(a)\cdot \res^S_{S'}(b)
		\end{equation*}
		is an isomorphism.
		
		\item For any $A$-algebra $C$ and $n\in \N$ the ring homomorphism
		\begin{equation*}
			W_{S/n}(C)	\otimes_{F_n,W_S(A)}W_S(B)\to W_{S/n}(C\otimes_AB),~
			c\otimes b\mapsto W_{S/n}(f)(c)\cdot F_n(b)
		\end{equation*}
		is an isomorphism.
	\end{enumerate}
\end{cor}
For an affine scheme $X=\spec R$ we define $W_S(X)\defeq \spec W_S(R)$. 
If $X$ is a separated scheme with an affine open covering $\bigcup_{i} U_i, U_i=\spec R_i$, then $W_S(X)$ is obtained by gluing the affine schemes $W_S(U_i)$ along the affine schemes $W_S(U_i\times_X U_j)$. It can be shown that this procedure does not depend on the covering.\\
If $X$ is an arbitrary scheme with affine open covering $\bigcup_{i} U_i$, then the open subschemes $U_i\times_XU_j$ are separated. We can glue the schemes $W_S(U_i)$ along the open subschemes $W_S(U_i\times_X U_j)$ to obtain the scheme $W_S(X)=\bigcup_{i} W_S(U_i)$. Again, this does not depend on the chosen covering.
\begin{theorem}[{\cite[Cor. 15.2, 15.3 \& 15.6, Prop. 16.4]{Bor15}}]\thlabel{closedimmersion}
	Let $S$ be a finite truncation set and $X$ a scheme. Then $W_S(X)$ is a scheme. Moreover, the following statements hold:
	\begin{enumerate}[(i)]
		\item If $(U_i)_i$ is an \'etale (open) covering of $X$, then $(W_S(U_i))_i$ is an \'etale (open) covering of $W_S(X)$.
		\item Let $S'\subseteq S$ be a truncation subset. The surjective restriction map $\res^S_{S'}$ on Witt vector rings induces a closed immersion $\iota_{S',S} \colon W_{S'}(X)\to W_S(X)$ of schemes.
	\end{enumerate}
	Let $X\to Y$ be a morphism of schemes.
	\begin{enumerate}
		\item[(iii)]If $X\to Y$ is an open/closed immersion, then the induced map $W_S(X)\to W_S(Y)$ is an open/closed immersion.
		\item[(iv)] If $X\to Y$ is \'etale, then the induced map $W_S(X)\to W_S(Y)$ is \'etale. If moreover $X'\to Y$ is an arbitrary morphism of schemes, then the map
		$$ (W_S(\pr_X),W_S(\pr_{X'}))\colon W_S(X\times_YX')\to W_S(X)\times_{W_S(Y)}W_S(X')$$
		is an isomorphism.
	\end{enumerate}	
\end{theorem}
\begin{lemma}[{\cite[10.6.8, Cor. 15.4 \& 15.7]{Bor15}}] \thlabel{ghostscheme}
Let $S$ be a finite truncation set. 
	\begin{enumerate}[(i)]
		\item  For every scheme $X$ the ghost map induces a surjective, integral morphism of schemes
		$$\gh_S \colon \coprod_{n\in S} X \to W_S(X).$$
		The projection onto the first component induces a closed immersion $\gh_1\colon X\to W_S(X)$.
		\item  Let $X\to Y$ be an \'etale morphism of schemes. For every $n\in S$ the $n$-th ghost component induces a cartesian diagram
		\begin{center}
			\begin{tikzcd}
				X\arrow[r,"\gh_n"]\arrow[d]&W_S(X)\arrow[d]\\
				Y\arrow[r,"\gh_n"]&W_S(Y).
			\end{tikzcd}
		\end{center}
	\end{enumerate}
\end{lemma}
\begin{definition}\thlabel{schemestorsionfree} \label{schemestorsionfreex}
	Given a truncation set $S$, a scheme $X$ is called $S$-\emph{torsion free} if there exists an affine open covering $X=\bigcup_i\spec R_i$, where each $R_i$ is an $S$-torsion free ring.\\ (Equivalently, $X$ is $S$-torsion free if multiplication by elements in $S$ is injective on $\OO_X$.)
\end{definition}

	Given a (locally) noetherian scheme $X$, the associated Witt scheme is not necessarily (locally) noetherian (see \cite[Counterexample 16.6]{Bor15}). 
To ensure this we need more information about the scheme $X$ with respect to a certain base ring. 
Indeed, Borger has defined Witt vector rings (and Witt schemes) relative to a given base ring and so-called supramaximal ideals generalizing the concept of truncation sets in \cite{Bor15a}. We only consider schemes $X$ over the base rings $\Z$, $\Z_{(p)}$ or $\Z_p$ for a fixed prime number $p$ and write $\Z_{\circ}$\label{Zcirc} as a placeholder for these three rings.\\
Given a truncation set $S$ and a $\Z_{\circ}$-algebra $R$, the supramaximal ideals in $\Z_{\circ}$ can be chosen in a way such that Borger's generalized Witt vector rings over $\Z_{\circ}$ agree with the classical ring of Witt vectors $W_S(R)$.
Given a scheme $X$ over $\Z_{\circ}$ and a finite truncation set $S$, \hyperref[Dwork'sLemma]{Dwork's Lemma} ensures the existence of a structure map $W_S(X)\to \spec \Z_{\circ}$. On an affine open subscheme $\spec R\subseteq X$ it corresponds to the ring homomorphism 
\begin{equation*}
	\Z_{\circ}\to W_S(\Z_{\circ})\to W_S(R),~
	z\mapsto \gh_S^{-1}((z,\dots,z)).
\end{equation*}
\begin{theorem}[{\cite[Prop. 16.5, 16.13 \& 16.19, Cor. 16.7]{Bor15}}]\thlabel{wittschemenl}
	Let $S$ be a finite truncation set and $X$ a $\Z_{\circ}$-scheme.
	\begin{enumerate}[(i)]
		\item  If $X$ is flat over $\Z_{\circ}$, then $W_S(X)$ is flat over $ \Z_{\circ}$.
		\item  If $X$ is (locally) of finite type over $ \Z_{\circ}$, then $W_S(X)$ is (locally) of finite type over $\Z_{\circ}$.
	\end{enumerate}
	From now on let $X$ be (locally) of finite type over $\Z_{\circ}$.
	\begin{enumerate}
		\item[(iii)]  Let $X\to Y$ be a morphism of $\Z_{\circ}$-schemes, where also $Y$ is (locally) of finite type over $\Z_{\circ}$. The induced map $W_S(X)\to W_S(Y)$ is (locally) of finite type.
		\item[(iv)] $W_S(X)$ is (locally) noetherian. If $X$ is quasi-compact, then $W_S(X)$ is noetherian.
	\end{enumerate}	
\end{theorem}
Combining this with \thref{closedimmersion} (iii) gives
\begin{cor}
	Let $S$ be a finite truncation set and $X$ a scheme. If $W_S(X)$ is (locally) noetherian, then $W_{S'}(X)$ is (locally) noetherian for every truncation subset $S'\subseteq S$. \thlabel{subschemenoetherian}
\end{cor}
In particular, if $W_S(X)$ is (locally) noetherian, then $X$ is (locally) noetherian.
\subsection[The relative de Rham-Witt complex]{\textbf{The relative de Rham-Witt complex.}}\hspace{0mm}
For every ring $R$ we have the \textit{absolute de Rham-Witt complex} $S\mapsto W_S\Omega^{\bullet}_R$ constructed by Hesselholt \cite{Hes14}. It is the initial object in the category of Witt complexes over $R$:
\begin{definition}
	A \textit{Witt complex} over $R$ is a contravariant functor
	\begin{align*}
		(\text{truncation sets})~~~~&\to ~~~~ (\text{anticommutative graded rings})\\
		S~~~~&\mapsto ~~~~ E_S^{\bullet},
	\end{align*}
	which maps colimits to limits together with a natural ring homomorphism ${\eta_S\colon W_S(R)\to  E_S^0}$ and natural maps of graded abelian groups
	$d\colon E_S^q\to  E_S^{q+1},~F_n\colon E_S^q\to  E_{S/n}^q, ~V_n\colon  E_{S/n}^q\to  E_S^q$
	for all $n\in \N$ and $q\geq 0$ satisfying a long list of compatibility relations.\\
	A \textit{morphism of Witt complexes} over $R$ is a natural map of graded rings $f\colon E_S^{\bullet}\to  {E_S'}^{\bullet}$
	satisfying $f\eta=\eta'f,~fd=d'f,~fF_n=F_n'f\text{ and }fV_n=V_n'f$
	for all $n\in\N$.
\end{definition}
In a Witt complex over $R$ the relations
$$dF_n= nF_nd\quad \text{and}\quad V_nd= ndV_n$$ 
hold for all $m,n\in \N$, i.e. neither the Frobenius nor the Verschiebung maps are morphisms of dga's.
Every Witt complex is uniquely determined -- up to isomorphism -- by its values on finite truncation sets $S$. Thus we often reduce our considerations to the finite case.\\
There is a unique morphism of graded rings
		$\Omega_{ W_S(R)}^{\bullet}\to W_S\Omega_{R}^{\bullet}$, which turns out to be surjective.
		Moreover, the map $$ \Omega_R^q\to W_{\{1\}}\Omega_R^q,~r_0dr_1\cdots dr_q\mapsto \eta_{\{1\}}(r_0)d\eta_{\{1\}}(r_1)\cdots d\eta_{\{1\}}(r_q)$$ is an isomorphism for every $q\geq 0$.
		For every truncation set $S$ the ring homomorphism $\eta_S\colon\hspace{-1mm} W_S(R)\hspace{-1mm}\to \hspace{-1mm}W_S\Omega^0_R$ is an isomorphism.\\

For a finite truncation set $S$ the functor $\spec R\mapsto W_S\Omega^q_R$ defines a presheaf of $W_S(\OO)$-modules on the category of affine schemes, where $W_S(\OO)$ denotes the sheaf of rings given by  $\spec R\mapsto W_S(R) $.
The following theorem tells us that this is indeed a quasi-coherent sheaf of $W(\OO)$-modules for the \'etale topology.

\begin{theorem}[{\cite[Thm. C]{Hes14}}]\label{absetbasechange}
	Given an \'etale ring homomorphism $f\colon A\to B$ and a finite truncation set $S$, the map
	$$\alpha\colon W_S\Omega^q_A\otimes_{W_S(A)}  W_S(B)\to W_S\Omega^q_B, ~\omega\otimes b \mapsto b\cdot W_S(f)(\omega)$$
	induces an isomorphism of Witt complexes.
\end{theorem}\vspace{-3mm}
We are even more interested in the relative setting:
\begin{theodef}[{\cite[Def. 1.2.1, Prop. 1.2.4]{Cha13}}]
	Let $R_0\to R$ be a ring homomorphism.
	The category of Witt complexes over $R$ with $W(R_0)$-linear differential has an initial object. 
	It is called the \emph{relative de Rham-Witt complex} and defined by the functor $$S\mapsto W_S\Omega_{R/R_0}^{q}\defeq \varprojlim_{S'\subseteq S} W_{S'}\Omega^{q}_R/(W_{S'}\Omega^1_{R_0}\cdot W_{S'}\Omega^{q-1}_R), q\geq 0,$$
	from truncation sets to dga's with $W(R_0)$-linear differential (or $W(R_0)$-dga's). Here the limit runs over the directed poset of finite truncation subsets $S'\subseteq S$. \label{reldRW}
\end{theodef}\vspace{-5mm}
For $q= 0$ the definition means $W_S\Omega^0_{R/R_0}=W_S(R)$. 

\begin{lemma}[{\cite[Lemma 1.2.2]{Cha13}}]\thlabel{surjreldrw}
	For finite truncation sets $S$ there is a unique map
	$$\alpha\colon \Omega_{ W_{S}(R)/W_S(R_0)}^{\bullet}\to W_S\Omega^{\bullet}_{R/R_0}$$
	of $W_S(R_0)$-dga's. It is surjective.
\end{lemma}
\vspace{-2mm}
The Witt complex structure on the absolute de Rham-Witt complex $W_S\Omega^{\bullet}_R$ induces the Witt complex structure on the relative version $W_S\Omega_{R/R_0}^{\bullet}$ (with $W(R_0)$-linear differential).
In particular, we again have Frobenius and Verschiebung operators for all $n\in\N$:
\begin{align*}
	F_n\colon &W_S\Omega_{R/R_0}^q\to W_{S/n}\Omega_{R/R_0}^q,\\
	V_n\colon &W_{S/n}\Omega_{R/R_0}^q\to W_S\Omega_{R/R_0}^q.
\end{align*}

\begin{remark}\thlabel{propertiesdRW}
	In a Witt complex over $R$ with $W(R_0)$-linear differential the following relations hold for all $m,n\in \N$ and $i,j\in\Z$ such that $im+jn=(m,n)$:
	\begin{gather*}
		dF_n=nF_nd,		~
		V_nd=ndV_n,\\
		F_mdV_n=idF_{m/(m,n)}V_{n/(m,n)}+jF_{m/(m,n)}V_{n/(m,n)}d,\\
		F_ndV_n=d,~
		V_nF_nd=dV_nF_n,~
		dV_nd=0.
	\end{gather*}
\end{remark}

From Theorem \ref{absetbasechange} we can deduce the following analogous behaviour of the relative de Rham-Witt complex with respect to \'etale base change.

\begin{cor}[{\cite[\nopp1.2.8 \& Lemma 1.2.9]{Cha13}}] \thlabel{etalechat}~
	\begin{enumerate}[(i)]
		\item  Given an \'etale ring homomorphism $f\colon R\to R'$ of $R_0$-algebras and a finite truncation set $S$, the induced map
		$$W_S\Omega^q_{R/R_0}\otimes_{W_S(R)}  W_S(R')\to W_S\Omega^q_{R'/R_0},~ \omega\otimes r'\mapsto r'\cdot W_S(f)(\omega)$$
		is an isomorphism.
		\item  Given an \'etale ring homomorphism $f\colon R_0\to R_0'$, an $R_0'$-algebra $R$ and an arbitrary truncation set $S$, the induced map
		$W_S\Omega^q_{R/R_0}\to W_S\Omega^q_{R/R_0'}$ is an isomorphism.
	\end{enumerate}
\end{cor}
Thus we obtain a sheafified relative de Rham-Witt complex as follows: 
\begin{prop}[{\cite[Prop. 1.2.11]{Cha13}}] 
	Let $R_0$ be a ring, $X$ an $ R_0$-scheme, $S$ a finite truncation set and $q\geq 0$. There is a unique quasi-coherent sheaf of $W_S(\OO)$-modules $W_S\Omega_{X/R_0}^q$ for the \'etale topology of $W_S(X)$ such that 
	$$W_S\Omega_{X/R_0}^q(\spec W_S(R))=W_S\Omega^q_{R/R_0}$$ 
	holds for every \'etale map $\spec R\to X$.
\end{prop}
\begin{cor}[{\cite[Prop. 1.2.12]{Cha13}}] 
	Let $S$ be a finite truncation set and $q\geq 0$.\\
	If $X\to \spec R_0$ is a morphism of schemes such that the induced morphism of Witt schemes $W_S(X)\to \spec W_S(R_0)$ is of finite type and $W_S(X)$ is noetherian, then $W_S\Omega_{ X/R_0}^q$ is coherent.
\end{cor}
Every morphism of schemes $f\colon X\to Y$ over $\spec R_0$ induces a morphism 
\begin{align*}
	W_S\Omega_{Y/R_0}^q\to W_S(f)_*W_S\Omega_{X/R_0}^q.
\end{align*}
Moreover, for every inclusion of truncation sets $S'\subseteq S$ we have a closed immersion $\iota_{S',S}\colon W_{S'}(X)\to W_S(X)$ functorial in $X$ which induces a morphism
\begin{align*}
	W_S\Omega_{X/R_0}^q\to \iota_{S',S*}W_{S'}\Omega_{X/R_0}^q.
\end{align*}
We obtain the differential, Frobenius and Verschiebung operators
\begin{align*}
	d\colon &W_S\Omega_{X/R_0}^q\to W_S\Omega_{X/R_0}^{q+1},\\
	F_n\colon &W_S\Omega_{X/R_0}^q\to \iota_{S/n,S*}W_{S/n}\Omega_{X/R_0}^q\\
	V_n\colon &\iota_{S/n,S*}W_{S/n}\Omega_{X/R_0}^q\to W_S\Omega_{X/R_0}^q
\end{align*}
via sheafification (see \cite[\nopp 1.2.13]{Cha13}). Furthermore, Langer and Zink's elaborated observation \cite[Cor. 2.18]{LZ04} that the relative de Rham-Witt complex of a $p$-torsion free polynomial ring is $p$-torsion free was generalized to smooth schemes by Chatzistamatiou:
\begin{prop}[{\cite[Prop. 1.2.17]{Cha13}}]\thlabel{chatsmooth}
	Let $S$ be a finite truncation set, $R_0$ a flat $\Z$-algebra and $X$ a scheme smooth over $R_0$-scheme. Then $W_S\Omega^q_{X/R_0}$ is $\Z$-torsion free, i.e. multiplication with a non-zero integer is injective. If moreover $X$ is smooth over $\spec R_0$ of relative dimension $d$, then $W_S\Omega^q_{X/R_0}=0$ for all $q>d$.
\end{prop}

\section[Torsionless differential forms]{\scshape{Torsionless differential forms}}
Motivated by \thref{chatsmooth} our strategy was to find a suitable definition of ``torsion'', which produces a ``torsion free'' functor compatible with \'etale base change, such that the relative de Rham-Witt complex is obtained as a ``torsion free version'' of the naive differential forms of the associated Witt scheme. The following example shows that the relative de Rham-Witt complex is not reflexive. 
\begin{example}
	For $S=\{1,2\}$ we show that  $W_S\Omega^{\bullet}_{\Z[t]/\Z}$ is not reflexive.
	By \cite[Lem. 3.11]{CD15} we can identify $W_S\Omega^{\bullet}_{\Z[t]/\Z}$ with $\Omega^{\bullet}_{W_S(\Z[t])/W_S(\Z)}/\Z\text{-torsion}$.
	An explicit computation shows that the $W_S(\Z[t])$-dual module in degree one is given by
	\begin{align*}
		&(W_S\Omega_{\Z[t]/\Z}^{1})^{\vee}
		\cong (\Omega_{W_S(\Z[t])/W_S(\Z)}^{1})^{\vee}
		\cong\text{Der}_{W_S(\Z)}(W_S(\Z[t]),W_S(\Z[t]))\\
		=&\left\{\begin{array}{c}
			d\colon W_S(\Z[t])\to W_S(\Z[t])\\
			W_S(\Z)\text{-linear derivation}
		\end{array}\Big|
		\begin{array}{l}
			d[t]=(2-V_2[1])[f]+V_2[gt]\\dV_2[t]=V_2[g]
		\end{array},\ f,g\in \Z[t]\right\}
	\end{align*}
	Let $\varphi_{f,g}\in (\Omega_{W_S(\Z[t])/W_S(\Z)}^{1})^{\vee}$ denote the $W_S(\Z[t])$-linear map corresponding to $f,g\in \Z[t]$.
	We identify $(W_S\Omega_{\Z[t]/\Z}^1)^{\vee\vee}$ with 
	$(\Omega_{W_S(\Z[t])/W_S(\Z)}^1)^{\vee\vee}$
	and show that the $W_S(\Z[t])$-linear map $\psi\colon \varphi_{f,g}\mapsto  (2-V_2[1])[f]+V_2[gt],f,g\in\Z[t],$ does not lie in the image of the evaluation map. \\
	Let us assume the existence of $\omega\in \Omega_{W_S(\Z[t])/W_S(\Z)}^1$ such that $\psi=\ev_{\overline{\omega}}$ with $\overline{\omega}=\omega+\Z\text{-torsion}$. 
	There are
	$a,b\in W_S(\Z[t])$ such that $\omega=ad[t] +bdV_2[t]$.
	We set  $a_n\defeq \gh_n(a)$, $b_n\defeq \gh_n(b)$ for $ n\in S$ and compute $\gh_S(\varphi_{1,1}(\overline{\omega}))=2(a_1,a_2t+b_2)$.
	This is equal to $\gh_S(\psi(\varphi_{1,1}))=2(1,t)$, hence implies $a_1=1$ and $a_2t+b_2=1$. 
	By \hyperref[Dwork'sLemma]{Dwork's Lemma} both $1-a_2$ and $b_1(t^2)-b_2
	=b_1(t^2)+a_2t-1
	$ lie in $2\Z[t]$. But we have
	\begin{equation*}
		b_1(t^2)-b_2
		=b_1(t^2)+t-1-(1-a_2)t
		\equiv b_1(t^2)+t-1\mod 2\Z[t]
	\end{equation*}
	and the degree $1$ term of $b_1(t^2)+t-1$ is $t\notin 2\Z[t]$, which contradicts the assumption.
\end{example}

We introduce the weaker property of being torsionless. This property does not appear very often in literature, but deserves to be better known.
\begin{definition}[{\cite[II, sec. 4]{Bas60}}]
	Let $M$ be an $R$-module.
	The coimage $T_R(M)\defeq M/\ker\epsilon_M$ of the canonical map $\epsilon_M\colon M\to  M^{\vee\vee}=\Hom_R(\Hom_R(M,R),R)$ is called the \textit{torsionless quotient}.
	The module $M$ is called \textit{torsionless} if $\epsilon_M\colon M\to  M^{\vee\vee}$ is injective or equivalently $T_R(M)=M$.
\end{definition}
\begin{prop}[{\cite[Rem. 4.65]{Lam12}}] \thlabel{torsionlesssub}Let $M$ be an $R$-module.
	\begin{enumerate}[label=(\roman*),topsep=-8pt]
		\item Torsionless modules are torsion free.
		\item $M$ is torsionless if and only if $M$ can be embedded into a direct product $R^I$ for some (possibly infinite) index set $I$.
		\item Any submodule of a free module is torsionless.
		\item Any submodule of a torsionless module is torsionless.
		\item Direct sums and direct products of torsionless modules are torsionless.
		\item If $R$ is an integral domain and $M$ a finitely generated torsion free $R$-module, then $M$ is torsionless.
	\end{enumerate}\nointerlineskip
\end{prop}
Forming the torsionless quotient defines a covariant endofunctor on the category of $R$-modules compatible with arbitrary direct sums. Moreover, the projections $M\to  T_R(M)$ constitute natural transformations from the identity functors to $T_R$. The functor $T_R$ is neither right exact nor preserves injectivity.\\
For a sheaf $\FF$ of $\OO_X$-modules on an arbitrary scheme $X$ again there is a natural ``evaluation map''
$\epsilon_{\FF}\colon \FF\to \FF^{\vee\vee}=\sHom_{\OO_X}(\sHom_{\OO_X}(\FF,\OO_X),\OO_X)$ of sheaves of $\OO_X$-modules. We define the \textit{torsionless quotient} $\T(\FF)\defeq\FF/\ker\epsilon_{\FF}$. 
\begin{definition}
	We call a sheaf $\FF$ of $\OO_X$-modules \textit{torsionless} if the evaluation map \\${\epsilon_{\FF}\colon \FF \to \FF^{\vee\vee}}$ is injective or equivalently $\T(\FF)=\FF$.
\end{definition}
	If $\FF$ is an $\OO_X$-module of finite presentation and stalkwise torsionless, then $\FF$ is a torsionless $\OO_X$-module by the commutativity of the diagram \begin{equation*}
		\begin{tikzcd}
			\FF_x\arrow[r]\arrow[d,hook]& (\FF^{\vee\vee})_x\arrow[d]\\
			(\FF_x)^{\vee\vee}\arrow[r,"\sim"]&\Hom_{\OO_{X,x}}(\sHom_{\OO_X}(\FF,\OO_X)_x,\OO_{X,x})
		\end{tikzcd}
	\end{equation*} for all $x\in X$. To make sure that the reverse statement is true we assume $X$ to be locally noetherian from now on.
	\begin{prop}\thlabel{cohsurjflat}
		Let $X$ be a locally noetherian scheme and $\FF$ a coherent $\OO_X$-module.
		\begin{enumerate}[label=(\roman*),topsep=-8pt]
			\item The torsionless quotient $\T(\FF)$ is a coherent $\OO_X$-module.
			\item For any flat morphism $f\colon Y\to X$ of schemes
			the canonical isomorphism\\ ${f^*(\FF^{\vee\vee})\to(f^*\FF)^{\vee\vee}}$ induces an isomorphism
			$$f^*\T(\FF)\to \T(f^*\FF )$$
			of finitely presented sheaves of $\OO_Y$-modules. 
		\end{enumerate}
	
			We have $\T(\FF)_x\cong T_{\OO_{X,x}}(\FF_x)$ for all $x\in X$.
		Moreover, let ${U=\spec R\subseteq X}$ be an affine open subscheme and $M$ an $R$-module associated with $\FF|_U$. Then $\T(\FF)|_U$ is associated with $T_R(M)$.
		The following are equivalent:
		\begin{enumerate}[topsep=-8pt]
			\item[(iii)] $\FF$ is a torsionless $\OO_X$-module.
			\item[(iv)] $\FF_x$ is a torsionless $\OO_{X,x}$-module for every $x\in X$.
			\item[(v)] $\FF(U)$ is a torsionless $\OO_{X}(U)$-module for all affine open subschemes $U\subseteq X$.
		\end{enumerate}\nointerlineskip\vspace{-2mm}
	\begin{proof}
		In this setting there is the natural isomorphism $f^*(\FF^{\vee\vee})\overset{\sim}{\to}f^*(\FF^{\vee})^{\vee}\overset{\sim}{\to}(f^*\FF)^{\vee\vee}$ and $\T(\FF)$ is quotient of coherent modules. The second statement follows from the exactness of flat pullback. The remaining statements are obvious by quasi-coherence.
	\end{proof}
	\end{prop}
\begin{cor}\thlabel{flatpullbacktorsionless}
	Let  $\FF$ be a coherent $\OO_X$-module on a locally noetherian scheme $X$. 
	\begin{enumerate}[label=(\roman*),topsep=-8pt]
		\item The torsionless quotient $\T(\FF)$ is torsionless.
		\item We have $\T(\T(\FF))\cong \T(\FF)$.
		\item If $\FF$ is a torsionless $\OO_Y$-module and $f\colon Y\to X$ a flat morphism, then $f^*\FF$ is a torsionless $\OO_Y$-module.
	\end{enumerate}\nointerlineskip\vspace{-2mm}
\begin{proof}
	Statement (i) and (ii) easily be checked locally. For part (iii) note that $f^*\FF$ is finitely presented. Therefore torsionless stalks guarantee a torsionless sheaf. Moreover, flat base change implies $T_{\OO_{Y,y}}((f^*\FF)_y)\cong T_{\OO_{X,f(y)}}(\FF_{f(y)})\otimes_{\OO_{X,f(y)}}\OO_{Y,y}\cong (f^*\FF)_y$ for all $y\in Y$.
\end{proof}
\end{cor}
These properties of the torsionless quotient on locally noetherian schemes enable us to prove
\begin{lemma}\thlabel{etalesheaf} 
	Let $f\colon X\to Y$ be a morphism of schemes and $q\geq 0$. If $X$ is locally noetherian and $f$ is locally of finite type, then the functor
	$$(U\to X)\mapsto \T(\Omega^q_{U/Y})(U)$$
	defines a coherent $\OO$-module for the \'etale topology of $X$.
\end{lemma}\vspace{-5mm}

Here $\OO$ denotes the sheaf of rings $(U\to X)\mapsto \OO_U(U)$ on the \'etale site of $X$.
\begin{proof}
	Given a morphism $f\colon X\to Y$ locally of finite type, the sheaf of relative differential forms $\Omega^q_{X/Y}$  is a quasi-coherent $\OO_X$-module of finite type (see \cite[\href{https://stacks.math.columbia.edu/tag/01V2}{01V2}, \href{https://stacks.math.columbia.edu/tag/01CK}{01CK}]{StacksProject}).
	As $X$ is locally noetherian, it is coherent and \thref{cohsurjflat} (i) applies.
	Thus $\T(\Omega_{ X/Y}^q)$ is a coherent $\OO_X$-module, too.\\
	Any descent datum given by an fpqc covering on quasi-coherent sheaves is effective (see \cite[\href{https://stacks.math.columbia.edu/tag/03NV}{03NV},   \href{https://stacks.math.columbia.edu/tag/023T}{023T}]{StacksProject}).
	Hence the functor 
	$(U\overset{g}{\to} X)\mapsto g^*\T(\Omega^q_{X/Y})(U)$
	defines a quasi-coherent sheaf on the fpqc site of $X$.
	By \thref{cohsurjflat} (ii) and the canonical exact sequence for relative differential forms we obtain
	$g^*\T(\Omega^q_{X/Y})
	\cong\T(g^*\Omega^q_{X/Y})
	\cong \T(\Omega^q_{U/Y}) 
	$
	if $g$ is \'etale.
\end{proof}
Beyond the structure of an $\OO$-module on the de Rham complex we also want to extend the structure of an exterior algebra. Later we will be able to obtain a differential graded algebra on the torsionless quotient by combining both. 
\begin{lemma}\thlabel{embext}
	Let $X$ be a locally noetherian scheme and $\FF$ a coherent $\OO_X$-module. 
	Assume that $f\colon X\to \spec \Z$ is flat and denote by $$\iota\colon X_{\Q}\defeq X\times_{\spec \Z}\spec \Q\cong f^{-1}(\{\eta\})\hookrightarrow X$$ the inclusion of the generic fibre.
	\begin{enumerate}[label=(\roman*),topsep=-8pt]
		\item The torsionless quotient $\T(\FF)$ can be interpreted as an $\OO_X$-submodule of $\iota_*\T(\FF|_{X_{\Q}})$ via the canonical morphism
		$\T(\FF)\to \iota_*\iota^*\T(\FF)$.
		\item If $\iota^*\FF$ is a finite locally free $\OO_{X_{\Q}}$-module, then $\bigoplus_{q\geq 0} \T(\Lambda^{q} \FF)$ inherits the wedge pro\-duct of the exterior algebra $\Lambda^{\bullet}\FF$, i.e. for all $q,r\geq 0$ the diagram 
		\begin{mycenter}[0mm]
			\begin{tikzcd}
				\Lambda^q\FF\times\Lambda^r\FF\arrow[r,"\wedge"]\arrow[d] 
				& \Lambda^{q+r}\FF\arrow[d]\\
				\T(\Lambda^q\FF\times \Lambda^r\FF)\arrow[r,"\T(\wedge)"]
				&\T(\Lambda^{q+r}\FF)
			\end{tikzcd}
		\end{mycenter}
		commutes and we obtain a map $\wedge\colon \T(\Lambda^q\FF)\times \T (\Lambda^r\FF)\cong\T(\Lambda^q\FF\times \Lambda^r\FF)\hspace{-1mm} \to\hspace{-1mm} \T(\Lambda^{q+r}\FF)$.
		Moreover, the canonical morphism $\bigoplus_{q\geq 0} \T(\Lambda^{q}\FF)\hspace{-1mm}\to\hspace{-1mm} \T(\Lambda^{\bullet}\FF)$ is an isomorphism of graded $\OO_X$-algebras.
	\end{enumerate}\nointerlineskip
	\begin{proof}
		First we note that $\iota\colon X_{\Q}\to X$ is an affine morphism, hence quasi-compact and separated (see \cite[ \href{https://stacks.math.columbia.edu/tag/01S7}{01S7}]{StacksProject}). Therefore $\iota_*\iota^*\T(\FF)$ is quasi-coherent (see \cite[ \href{https://stacks.math.columbia.edu/tag/03M9}{03M9}]{StacksProject}). 
		On an affine open subscheme $U=\spec R\subseteq X$, where $\FF|_U$ is associated with an $R$-module $M$, the injectivity of the morphism restricted to $U$ is equivalent to the injectivity of the homomorphism
		$T(M)\to T(M)\otimes_R R_{\Q}\cong T(M)_{\Q}$ of $R$-modules. This is clear because $T(M)$ is a flat $\Z$-module. Moreover, by \thref{cohsurjflat} (ii) we globally have an isomorphism $\iota^*\T(\FF)\cong \T(\iota^*\FF)=\T(\FF|_{X_{\Q}})$.\\
		For (ii) note that $\iota^*\FF$ is finitely presented and locally free, hence finite locally free (see \cite[ \href{https://stacks.math.columbia.edu/tag/05P2}{05P2}]{StacksProject}). 
		Recall that the exterior powers of a finite locally free module are itself finite locally free (see \cite[ \href{https://stacks.math.columbia.edu/tag/01CK}{01CK (6)}]{StacksProject}). Indeed, also the exterior algebra $\Lambda^{\bullet} (\iota^*\FF)\cong \iota^*(\Lambda^{\bullet} \FF)$ is a locally free $\OO_{X_{\Q}}$-algebra on the locally noetherian scheme ${X_{\Q}}$ (see \cite[ \href{https://stacks.math.columbia.edu/tag/01CI}{01CI}, \href{https://stacks.math.columbia.edu/tag/01CL}{01CL}]{StacksProject}). 
		Therefore part (i) gives us an embedding of $\OO_X$-modules
		$$\T(\Lambda^q \FF)
		\hookrightarrow  \iota_*\T((\Lambda^q \FF)|_{X_{\Q}}))
		\overset{\sim}{\to}\iota_*((\Lambda^q \FF)|_{X_{\Q}})
		\overset{\sim}{\to}\iota_*(\Lambda^q (\iota^*\FF)).$$
		The fact that we are working on a locally noetherian scheme and the compatibility of the exterior algebra with pull back allow us to reduce to the affine setting. We assume $X=\spec R$ and choose an $R$-module $M$ associated with the coherent sheaf $\FF=\FF|_{\spec R}$. By assumption the base change $M_{\Q}$ is a finite locally free $R_{\Q}$-module and so is $\Lambda_{R_{\Q}}^{q} M_{\Q}$ for every $q\geq 0$. Hence these modules are reflexive, in particular torsionless.
		We consider $\omega\in \Lambda^q M$ and $\eta\in \Lambda^r M$ for some $q,r\geq 0$. 
		The image of the pair ${(\omega + \ker\epsilon_{q},\eta + \ker\epsilon_{r})}$ under the composition
		$$T(\Lambda^q M)\times T(\Lambda^r M)\hookrightarrow
		\Lambda_{R_{\Q}}^q M_{\Q}\times \Lambda_{R_{\Q}}^r M_{\Q}\overset{\wedge}{\to}
			\Lambda_{R_{\Q}}^{q+r} M_{\Q}\overset{\sim}{\to}
				(\Lambda^{q+r} M)_{\Q}$$ 
		is $(\omega\wedge\eta)\otimes 1$, which
		has a unique preimage in $T(\Lambda^{q+r} M)\hookrightarrow	(\Lambda^{q+r} M)_{\Q}$. Thus the map
		\begin{align*}
			\wedge \colon T(\Lambda^q M)\times T(\Lambda^r M)~~~~~
			&\to  T(\Lambda^{q+r} M)\\
			(\omega+ \ker\epsilon_{q},\eta+ \ker\epsilon_{r})
			&\mapsto  \omega\wedge\eta + \ker\epsilon_{q+r}
		\end{align*}
		is well-defined. \\
		For the last statement note that for every $q\geq0 $ we have a morphism $\Lambda^{q}\FF\to \Lambda^{\bullet}\FF$, which induces a morphism $\T(\Lambda^{q}\FF)\to\T( \Lambda^{\bullet}\FF)$ on the torsionless quotients. By taking the direct sum over all $q\geq 0$ we obtain the canonical morphism
		$\bigoplus_{q\geq 0}\T(\Lambda^{q}\FF)\to\T( \Lambda^{\bullet}\FF)$ of $\OO_X$-modules.
		Since $X$ is locally noetherian, $\Lambda^q\FF$ is a coherent $\OO_X$-module for every $q\geq 0$ (see \cite[ \href{https://stacks.math.columbia.edu/tag/01CK}{01CK (4)}]{StacksProject}). The direct sum commutes with the torsionless quotient in this setting and we obtain an isomorphism.
	\end{proof}
\end{lemma}

\begin{definition}\thlabel{gensmoothdef}
	Let $X\to Y$ be a morphism locally of finite type of flat $\Z$-schemes, where $X$ is locally noetherian.
	We call $X\to Y$ \textit{generically smooth} if the base change $ X_{\Q}\to  Y_{\Q}$ is smooth.
\end{definition}
We call a homomorphism of rings generically smooth if this is the case for the associated morphism of affine schemes.
Whenever we mention a generically smooth morphism, we silently put the assumptions of the definition.
\phantomsection\label{dgaIntro}
\begin{lemmaA}\thlabel{dga}
	Given a generically smooth morphism $f\colon X\to Y$ of schemes, the torsionless differential forms $\T(\Omega^{\bullet}_{X/Y})$ form an $f^{-1}\OO_Y$-dga.
\end{lemmaA}
\begin{proof}
	Since for each $q\geq 0$ the $\OO_X$-module $\Omega^{q}_{X/Y}$ is coherent by \cite[ \href{https://stacks.math.columbia.edu/tag/01V2}{01V2}, \href{https://stacks.math.columbia.edu/tag/01XZ}{01XZ}]{StacksProject}, we can apply part (i) of the previous lemma to embed the $\OO_X$-module $\T(\Omega^{q}_{X/Y})$ into $\iota_*\T(\Omega^{q}_{X_{\Q}/Y_{\Q}})$, where $\iota\colon X_{\Q}\to X$ is the inclusion of the generic fibre.
	But the $\OO_{X_{\Q}}$-module ${\iota^*\Omega^{q}_{X/Y}\cong \Omega^{q}_{X_{\Q}/Y_{\Q}}}$
	is finite locally free because the homomorphism $X_{\Q}\to  Y_{\Q}$ is smooth (see \cite[ \href{https://stacks.math.columbia.edu/tag/02G1}{02G1}]{StacksProject}). Therefore the second part of the lemma applies and we obtain $\T(\iota^*\Omega^{q}_{X/Y})=\iota^*\Omega^{q}_{X/Y}$.\\
	Though the differential $d\colon \Omega_{X/Y}^q\to \Omega_{X/Y}^{q+1}$ is not a map of $\OO_X$-modules but of $f^{-1}\OO_Y$-modules (see \cite[ \href{https://stacks.math.columbia.edu/tag/07HX}{07HX}]{StacksProject}), it still induces a map on the torsionless quotient:
	The derivation
	$d_{\Q}\colon \iota^*\Omega^{q}_{X/Y}\to  \iota^*\Omega^{q+1}_{X/Y}$ 
	maps $\iota^*\omega$ to $\iota^*(d\omega) $ on local sections. 
	Hence the graded $\OO_X$-algebra $\bigoplus_{q\geq 0}\T(\Omega^q_{X/Y})\cong\T(\Omega^{\bullet}_{X/Y})$ inherits all properties of a differential graded $f^{-1}\OO_Y$-algebra from $\Omega^{\bullet}_{X/Y}$ by similar arguments as in the proof of \thref{embext} part (ii). 
\end{proof}
\begin{remark}\thlabel{etalebottom}
	Let $X\overset{f}{\to} Y\overset{g}{\to}  Z$ be a pair of morphisms of schemes where $g$ is \'etale. 
	The induced map
	$\T(\Omega_{ X/Z}^q)\to \T(\Omega_{X/Y }^q)$
	is an isomorphism of $\OO_X$-modules.\\
	This follows from the canonical exact sequence
		$$0\to f^*\Omega_{ Y/Z}^q
		\to \Omega_{ X/Z}^q
		\to \Omega_{X/Y }^q
		\to 0$$
		with $f^*\Omega_{ Y/Z}^q=0$ and the application of the functor $\T$ to the isomorphism $\Omega_{ X/Z}^q
		\to \Omega_{X/Y }^q$.
\end{remark}

\section[Witt complexes as torsionless quotients]{\scshape{Witt complexes as torsionless quotients}}\label{sectionwittcomplexes}
In this section we establish the structure of a Witt complex on the torsionless quotient of the naive de Rham complex of Witt schemes. We assume all truncations sets to be finite if not specified otherwise. 
For the following considerations we need a certain finiteness condition on the underlying morphism of schemes:

\begin{definition} \thlabel{Wfinite}
A morphism of schemes $X\to Y$ is called $W_S\hspace{0.2mm}$\textit{-finite} (resp. $W$\textit{-finite}) if $W_S(X)$ is locally noetherian and the induced morphism $W_S(X)\to W_S(Y)$ of Witt schemes is locally of finite type for a fixed (resp. for every) finite truncation set $S$.
\end{definition}
\vspace{-2mm}
We call a homomorphism $R_0\to R$ of rings $W_S\hspace{0.2mm}$-finite (or $W$-finite) if the associated morphism of schemes is $W_{S}\hspace{0.2mm}$-finite (or $W$-finite).
This finiteness condition is satisfied in the commonly studied situations (see \thref{Zcircfinite}).\\\vspace{-2mm}

In this section we define restriction, Frobenius and Verschiebung operators on the functor 
\begin{align*}
	\begin{aligned}
		(\text{finite truncation sets})&\to (W(R_0)\text{-dga's})\\
		S&\mapsto T_{W_S}(\Omega_{R/R_0}^{\bullet})\defeq T_{W_{S}(R)}(\Omega_{W_{S}(R)/W_{S}(R_0)}^{\bullet})
	\end{aligned}
\end{align*}
with respect to a generically smooth, $W$-finite homomorphism $R_0\to R$ of rings. We denote by\vspace{-2mm}
$$\epsilon_{\Omega W_S}\colon \Omega_{W_{S}(R)/W_{S}(R_0)}^{q}\to (\Omega_{W_{S}(R)/W_{S}(R_0)}^{q})^{\vee\vee}$$
the evaluation map of the $W_S(R)$-module $\Omega_{W_{S}(R)/W_{S}(R_0)}^{q}$.
During our investigation we will observe that the structural maps like the Frobenius and Verschiebung maps on $T_{W_S}(\Omega_{R/R_0}^{\bullet})$ are directly inherited from the underlying ring of Witt vectors $W_S(R)$ (up to multiplication with constants).
At the second step we slightly adapt the setting to obtain the de Rham-Witt complex, i.e. we are going to prove the \hyperref[MAINTHEOREM]{main theorem} of this paper in section \ref{sectionreldrwcomplexes}.\\

Every Witt complex is uniquely characterized by its values on finite truncation sets. For an infinite truncation set we take the projective limit of dga's on the right hand side with respect to the poset of finite truncation subsets.

\phantomsection\label{wittcomplexintro}
\begin{theoremB}\thlabel{wittcomplex}
	For a generically smooth, $W$-finite ring homomorphism $R_0\to  R$ the functor
	\begin{equation*}
		(\text{finite truncation sets})\to (W(R_0)\text{-dga's}),~
		S\mapsto T_{W_{S}}(\Omega_{R/R_0}^{\bullet})
	\end{equation*}
	defines a Witt complex over $R$ with $W(R_0)$-linear differential.\\
	For an  \'etale homomorphism $R\to R'$ the induced map
	\begin{align*}
		T_{W_S}(\Omega^q_{R/R_0})\otimes_{W_S(R)}W_S(R') &\to T_{W_S}(\Omega^q_{R'/R_0})\\ (\omega+\ker\epsilon_{\Omega W_{S}})\otimes r'&\mapsto r'\cdot W_S(f)(\omega)+\ker\epsilon'_{\Omega W_{S}}
	\end{align*}
	is an isomorphism.
\end{theoremB}

We prove \thref{wittcomplex} by introducing the multiple structure maps stepwise for finite truncation sets. 
Afterwards we extend the operators to infinite truncation sets by showing the compatibility with the projective limit. The main tool of the proof is the following lemma.

\begin{lemma}
	\thlabel{embedding}
	Let $R_0\to R$ be a generically smooth, $W_S\hspace{0.2mm}$-finite ring homomorphism. There is an embedding $ T_{W_{S}}(\Omega_{R/R_0}^{\bullet})\hookrightarrow  (\Omega_{R_{\Q}/R_{0,\Q}}^{\bullet})^S$ of $W_S(R_0)$-dga's induced by the injective ghost maps on $R_0$ and $R$.
\end{lemma}\vspace{-3mm}
Note that we have an isomorphism $(R_{\Q})^S\cong (R^S)_{\Q}$ as $S$ is a finite set.\vspace{-3mm}
\begin{proof}
	Tensoring the respective ghost maps with the identity on $\Q$ induces the isomorphisms $W_S(R_0)_{\Q}\cong R^S_{0,\Q}$ and 
	$W_S(R)_{\Q}\cong R^S_{\Q}$
	by \thref{tensorqsghost} (i).
	By the definition of generical smoothness both $R_0$ and $R$ are flat $\Z$-algebras, hence also $W_S(R_0)$ and $W_S(R)$ are flat $\Z$-algebras again by \thref{tensorqsghost} (i). 
	Since the map $W_S(R_0)\to W_S(R)$ is of finite type, the module of K\"ahler differential forms $\Omega_{W_S(R)/W_S(R_0)}^{q}$ is finitely generated (see \cite[ \href{https://stacks.math.columbia.edu/tag/00RZ}{00RZ}, \href{https://stacks.math.columbia.edu/tag/01CK}{01CK (2)}]{StacksProject}). Moreover, the Witt vector ring $W_S(R)$ is noetherian by assumption. Hence $\Omega_{W_S(R)/W_S(R_0)}^{q}$ is coherent and the homomorphism $W_S(R_0)\to W_S(R)$ is generically smooth.
	By \thref{embext} (i) and \thref{Wittlocali} we obtain an embedding
	$$ T_{W_{S}}(\Omega_{R/R_0}^{q})
	\hookrightarrow T_{W_{S}}(\Omega_{R_{\Q}/R_{0,\Q}}^{q})
	\overset{\sim}{\to}T_{R^S_{\Q}}\left((\Omega_{R_{\Q}/R_{0,\Q}}^{q})^S\right)$$
	of $W_S(R)$-modules for each $q\geq 0$.
	The smoothness of the homomorphism $R_{0,\Q}^S\to  R_{\Q}^S$ implies that the associated module of K\"ahler differential forms is a finite locally free $ R_{\Q}^S$-module. 
	Hence it is reflexive and, in particular, it is torsionless. 
	Finally, we obtain the embedding
	$ T_{W_{S}}(\Omega_{R/R_0}^{q})\hookrightarrow  (\Omega_{R_{\Q}/R_{0,\Q}}^{q})^S$
	of $W_S(R)$-modules given by
	$$r_0 dr_1\cdots dr_q + \ker \epsilon_{\Omega W_S}\mapsto  (\gh_S(r_0)\otimes 1)d(\gh_S(r_1)\otimes 1)\cdots  d(\gh_S(r_q)\otimes 1)$$
	on equivalence classes of elementary differential forms $r_0 dr_1\cdots  dr_q\in \Omega_{W_S(R)/W_S(R_0)}^{q}$. By \thref{dga} this embedding is compatible with the differential, hence an embedding of $W_S(R_0)$-dga's.
\end{proof}

\begin{proof}[Proof of \thref{wittcomplex}]
	In the previous proof we have seen that the induced homomorphism $W_S(R_0)\to W_S(R)$ is generically smooth.
	Since $W_S(R)$ is noetherian by assumption, the torsionless quotient $T_{W_{S}}(\Omega_{R/R_0}^{\bullet})$ is a  $W_S(R_0)$-dga by \thref{dga}.
	In degree zero we trivially have the natural isomorphism
	$\eta_S\colon T_{W_S(R)}(W_S(R))\to W_S(R)$
	of $W_S(R_0)$-algebras.\\
	
	\textit{The Frobenius:}\phantomsection\label{constrfrobintro}
	For every $n\in S$ the Frobenius homomorphism $F_n$ is a natural transformation, hence induces a $W_{S}(R)$-linear map 
	$$F_n\colon \Omega_{W_{S}(R)/W_{S}(R_0)}^{\bullet}\to \Omega_{W_{S/n}(R)/W_{S/n}(R_0)}^{\bullet}, r_0dr_1\cdots dr_q\mapsto F_n(r_0)dF_n(r_1)\cdots dF_n(r_q), q\geq 0.\footnote{There is a Frobenius map on $\Omega_{ W_{S}(R)/W_S(R_0)}^q$ already, which induces the correct Frobenius map on $W_S\Omega_{ R/R_0}^q$. Its definition requires a certain natural ring homomorphism $\Delta_R\colon W(R)\to W(W(R))$, which we do not establish here (see \cite[Prop. 1.19, Thm. 2.15]{Hes14}).} $$
	In degree zero it is the usual Witt vector Frobenius, which turns $W_{S/n}(R)$ into a $W_S(R)$-algebra.
	To see that it induces a homomorphism on the torsionless quotient observe that any $\varphi\in \Hom_{W_{S/n}(R)} (\Omega_{W_{S/n}(R)/W_{S/n}(R_0)}^{q},W_{S/n}(R))$ induces a morphism $V_n\circ \varphi \circ F_n\in \Hom_{W_S(R)}(\Omega_{W_{S}(R)/W_{S}(R_0)}^{q}, W_{S}(R))$ by the projection formula satisfying $(V_n\circ \varphi\circ F_n)(\ker \epsilon_{\Omega W_{S}})$. Since $V_n\colon W_S(R)\to W_{S/n}(R)$ is injective, we obtain $(\varphi\circ F_n)(\ker \epsilon_{\Omega W_{S}})=0$. Unfortunately, the naive map
	\begin{align*}
		F_n\colon T_{W_{S}}(\Omega_{R/R_0}^{q})\quad \quad ~~~~~~~~~
		&\to T_{W_{S/n}}(\Omega_{R/R_0}^{q})\\
		r_0dr_1\cdots dr_q+\ker\epsilon_{\Omega W_S}
		&\mapsto F_n(r_0)dF_n(r_1)\cdots dF_n(r_q)+\ker\epsilon_{\Omega W_{S/n}}
	\end{align*}
	commutes with the differential instead of satisfying $dF_n=nF_nd$. We need to modify $F_n$. Namely, we want to set $F_n^q\defeq n^{-q}F_n$ in degree $q\geq0 $.
	We need to show that every element in the image of $F_n\colon T_{W_{S}}(\Omega_{R/R_0}^{1})\to T_{W_{S/n}}(\Omega_{R/R_0}^{1})$ is divisible by $n$. 
Since any Witt vector is a convergent sum of elements $V_k[r], r\in R, k\in S$, we can restrict to classes represented by elements $dF_nV_k[r],r\in R,k\in S$ (see \cite[Lemma 1.5]{Hes14}). An easy induction allows us only to consider prime number $n=p\in S$. We compute $	dF_pV_k[r]=	pdV_{k/p}[r]$ for $p|k$. Also for $p\nmid k$ the term
	\begin{equation*}
		dF_pV_k[r]
		=	dV_k[r^p]
		=dV_k[r^p]-d(V_k[r])^p+d(V_k[r])^p
		=(1-k^{p-1}) dV_k[r^p]+p(V_k[r])^{p-1}dV_k[r]
	\end{equation*}
	is a multiple of $p$. 
	We obtain the well-defined map
	\begin{align*}
		F_n^{q}\colon T_{W_{S}}(\Omega_{R/R_0}^{q})\to T_{W_{S/n}}(\Omega_{R/R_0}^{q}),~
		\omega+\ker\epsilon_{\Omega W_S}
		\mapsto n^{-q}F_n(\omega)+\ker\epsilon_{\Omega W_{S/n}}
	\end{align*}
	of graded abelian groups, which satisfies $dF_n^{q}=nF_n^{q+1}d$.\\
	
	\textit{The restriction:}
	For a  truncation subset $S'\subseteq S$ the natural restriction maps of the Witt vector rings induce the homomorphism 
	$$\res^S_{S'}\colon \Omega_{W_S(R)/W_S(R_0)}^{\bullet}\to  \Omega_{W_{S'}(R)/W_{S'}(R_0)}^{\bullet}$$
	of $W_S(R_0)$-dga's. 
	We have to show that this restriction respects the kernel of the canonical map into the respective biduals.
	The restriction of coefficients $\res^S_{S'}\colon R^S\to R^{S'}$ on $R^S$ (and analogously on $R_0^S$) induces the homomorphism
	$$\res^S_{S'}\colon (\Omega_{R_{\Q}/R_{0,\Q}}^{q})^S\to (\Omega_{R_{\Q}/R_{0,\Q}}^{q})^{S'}$$
	of $R_0^S$-modules.
	By \thref{embedding} we can identify the equivalence class of an elementary differential form $r_0 dr_1\cdots dr_q$ in $ T_{W_{S}}(\Omega_{R/R_0}^{q})$ with the element
	 $$(\gh_S(r_0)\otimes 1)d(\gh_S(r_1)\otimes 1)\cdots d(\gh_S(r_q)\otimes 1)\in(\Omega_{R_{\Q}/R_{0,\Q}}^{q})^S.$$ 
	By the naturality of the ghost map, applying the restriction yields
	 \begin{align*}
		&(\res^S_{S'}(\gh_S(r_0))\otimes 1)d(\res^S_{S'}(\gh_S(r_1))\otimes 1)\cdots d(\res^S_{S'}(\gh_S(r_q))\otimes 1)\\
		=&(\gh_{S'}(\res^S_{S'}(r_0))\otimes 1)d(\gh_{S'}(\res^S_{S'}(r_1))\otimes 1)\cdots d(\gh_{S'}(\res^S_{S'}(r_q))\otimes 1)
	\end{align*}
	in $(\Omega_{R_{\Q}/R_{0,\Q}}^{q})^{S'}$ which has the unique preimage
	 $\res^S_{S'}(r_0)d\res^S_{S'}(r_1)\cdots d\res^S_{S'}(r_q)+\ker\epsilon_{\Omega W_{S'}}$
	in $T_{W_{S'}}(\Omega_{R/R_0}^{q})$. 
	We thus have shown the well-definedness of the restriction on the torsionless quotient.\\
	
	\textit{The Verschiebung:}
	For the construction of the Verschiebung on the torsionless quotient note that the Verschiebung on the Witt vector rings  $V_n\colon W_{S/n}(R)\to W_S(R)$ for $n\in S$ is not only additive. 
	By tensoring with $\Q$ we get a ring homomorphism 
	$\frac{1}{n}V_n\colon W_{S/n}(R)_{\Q}\to W_S(R)_{\Q}$.
	After applying the ghost map both on the source and on the target we obtain the induced homomorphism of graded abelian groups
	\begin{equation*}
		V_n^{q}\defeq n^{q+1}\cdot \frac{1}{n}V_n\colon (\Omega_{R_{\Q}/R_{0,\Q}}^{q})^{S/n}
		\to (\Omega_{R_{\Q}/R_{0,\Q}}^{q})^{S/n}
	\end{equation*}
	given by 
	\begin{align*}
		r_0dr_1\cdots dr_q
		\mapsto ~~&n^{q+1}\cdot \frac{1}{n}V_n(r_0)d\frac{1}{n}V_n(r_1)\cdots d\frac{1}{n}V_n(r_q)
		=V_n(r_0)dV_n(r_1)\cdots dV_n(r_q)
	\end{align*}
	for $r_0,\dots,r_q\in R^{S/n}_{\Q}$. We now follow the same strategy as for the definition of the restriction maps, i.e. we show that the right perpendicular map in the diagram
	\begin{mycenter}[-2mm]
		\begin{tikzcd}
			T_{W_{S/n}}(\Omega_{R/R_0}^{q})\arrow[r,hook]\arrow[d,dashed, "V_n^q"']
			&(\Omega_{R_{\Q}/R_{0,\Q}}^{q})^{S/n}\arrow[d, "V_n^q"]\\
			T_{W_{S}}(\Omega_{R/R_0}^{q})\arrow[r,hook]
			&(\Omega_{R_{\Q}/R_{0,\Q}}^{q})^{S}
		\end{tikzcd}
	\end{mycenter} 
	restricts to the left perpendicular map.
	Using the upper horizontal inclusion $V_n^q$ maps the equivalence class of
	$r_0 dr_1\cdots dr_q$ in $ T_{W_{S/n}}(\Omega_{R/R_0}^{q})$ to
	$$(V_n(r_0)\otimes 1) d(V_n(r_1)\otimes 1)\cdots d(V_n(r_q)\otimes 1),$$
	which -- under the lower horizontal inclusion -- has the unique preimage
	$$V_n(r_0) dV_n(r_1)\cdots dV_n(r_q)+\ker\epsilon_{\Omega W_S}\in T_{W_{S}}(\Omega_{R/R_0}^{q}).$$
	By \thref{dga} the map
	$d\colon 	T_{W_{S}}(\Omega_{R/R_0}^{q})\to 	T_{W_{S}}(\Omega_{R/R_0}^{q+1})$
	is a $W_S(R_0)$-linear derivation satisfying $d\circ d =0$. The remaining properties of a Witt complex (in the relative setting) can be checked by straightforward computations (see \cite[Definition 4.1]{Hes14}).\\
	
	\textit{The projective limit:} An arbitrary (possibly infinite) truncation set $\overline{S}\subseteq \N$  is the direct limit of the poset of its finite truncation subsets $S\subseteq \overline{S}$. 
	As the functor $S\mapsto T_{W_{S}}(\Omega_{R/R_0}^{\bullet})$ is contravariant, the only thing that is left to show is the compatibility of all maps constructed above with the projective limit. We omit the details as the proof is straightforward.\\ 
	For $n\in \N$ as well as arbitrary truncation sets $\overline{S}$ and $\overline{S}_2\subseteq \overline{S}_1$ the maps are given as follows:
	\begin{align*}
		\res^{\overline{S}_1}_{\overline{S}_2}\colon
		& \varprojlim_{S\subseteq \overline{S}_1}T_{W_{S}}(\Omega_{R/R_0}^{\bullet})
		\to \varprojlim_{\substack{S\subseteq \overline{S}_1} }T_{W_{S\cap \overline{S}_2}}(\Omega_{R/R_0}^{\bullet})
		\overset{\sim}{\to} \varprojlim_{S\subseteq \overline{S}_2}T_{W_{S}}(\Omega_{R/R_0}^{\bullet})\\
		&(\omega_S)
		\mapsto(\res^S_{S\cap \overline{S}_2}(\omega_S))
		\mapsto(\omega_S)\\
		F_n^{\bullet}\colon
		& \varprojlim_{S\subseteq \overline{S}}T_{W_{S}}(\Omega_{R/R_0}^{\bullet})
		\to \varprojlim_{n\in S\subseteq \overline{S}}T_{W_{S/n}}(\Omega_{R/R_0}^{\bullet})
		\overset{\sim}{\to }\varprojlim_{S\subseteq \overline{S}/n}T_{W_{S}}(\Omega_{R/R_0}^{\bullet})\\
		&(\omega_S)
		\mapsto(F_n^{\bullet}(\omega_S))
		\mapsto(F_n^{\bullet}(\omega_{\la n\cdot S\ra}))\\
		V_n^{\bullet}\colon&	\varprojlim_{S\subseteq \overline{S}/n}T_{W_{S}}(\Omega_{R/R_0}^{{\bullet}})
		\to
		\varprojlim_{n\in S\subseteq \overline{S}}T_{W_{S}}(\Omega_{R/R_0}^{{\bullet}})
		\overset{\sim}{\to} 
		\varprojlim_{S\subseteq \overline{S}}T_{W_{S}}(\Omega_{R/R_0}^{{\bullet}})\\
		&(\omega_S)\mapsto (V_n^{\bullet}(\omega_{S/n}))
		\mapsto (\delta_{n\in S}V_n^{\bullet}(\omega_{S/n})).
	\end{align*}
	Note that any finite truncation subset $S\subseteq \overline{S}/n$ satisfies $S=\la n\cdot S\ra/n$ and $\la n\cdot S\ra $ is a finite truncation subset of $\overline{S}$, where $\la n\cdot S\ra$ denotes the set $\{m\mid m|nk\text{ for some }k\in S\}$. The differential is applied componentwise. \\
	
	\textit{\'Etale base change:} For an \'etale ring homomorphism $R\to R'$ and the induced homomorphism $W_S(R)\to W_S(R')$ is \'etale by \thref{BorKal}. By \thref{cohsurjflat} part (ii) the base change morphism on the torsionless quotient is an isomorphism as $W_S(R)$ is assumed to be noetherian. We have an isomorphism of $W_S(R_0)$-dga's 
	$$T_{W_S}(\Omega^{\bullet}_{R/R_0})\otimes_{W_S(R)} W_S(R') \overset{\sim}{\to} T_{W_S}(\Omega^{\bullet}_{R'/R_0}),$$ 
	where the left hand side is equipped with the differential
	\begin{align*}
		d\colon T_{W_S}(\Omega^{q}_{R/R_0})\otimes_{W_S(R)} W_S(R')& \to T_{W_S}(\Omega^{q+1}_{R/R_0})\otimes_{W_S(R)} W_S(R')\\
		\omega \otimes r'&\mapsto d\omega \otimes r'+(-1)^q\omega\otimes dr'.
	\end{align*} 
	Note that the structure of a Witt complex over $R'$ with $W(R_0)$-linear diffe\-rential on $T_{W_{S}}(\Omega_{R/R_0}^{\bullet})\otimes_{W_S(R)}W_S(R')$ can be described explicitly by transferring Hesselholt's description in the proof of \cite[Thm. C]{Hes14}).
\end{proof}

\begin{cor}
	For a generically smooth, $W_S\hspace{0.2mm}$-finite morphism $X\to \spec R_0$ the torsionless quotient $\T_{W_S}(\Omega_{X/R_0}^q)$ is a coherent $W_S(\OO)$-module for the \'etale topology of $W_S(X)$.\\
	It is equipped with the following structure maps:
	Every morphism $f\colon X\to Y$ of generically smooth, $W_S\hspace{0.2mm}$-finite $R_0$-schemes induces a morphism
	$$\T_{W_S}(\Omega^{q}_{Y/R_0}) \to W_S(f)_*\T_{W_S}(\Omega^{q}_{X/R_0}).$$
	For any truncation subset $S'\subseteq S$ we obtain the morphism
	$$\T_{W_S}(\Omega^{q}_{X/R_0}) \to \iota_{S',S*}\T_{W_{S'}}(\Omega^{q}_{X/R_0}),$$
	which is compatible with morphisms of schemes as above.\\
	Also the differential, Frobenius and Verschiebung maps extend to the level of sheaves
	\begin{align*}
		d\colon &\T_{W_S}(\Omega^{q}_{X/R_0})\to \T_{W_S}(\Omega^{q+1}_{X/R_0})\\
		F_n\colon& \T_{W_S}(\Omega^{q}_{X/R_0})\to\iota_{S/n,S*}\T_{W_{S/n}}(\Omega^{q}_{X/R_0})\\
		V_n\colon&\iota_{S/n,S*}\T_{W_{S/n}}(\Omega^{q}_{X/R_0})\to\T_{W_S}(\Omega^{q}_{X/R_0}).
	\end{align*}
\end{cor}
The maps above are constructed analogously to the affine case by using that the push forward is left exact (see \cite[ \href{https://stacks.math.columbia.edu/tag/01AJ}{01AJ}]{StacksProject}).
\begin{cor}\thlabel{etalebottomsmooth} Let $ X \to \spec R_0$ be a generically smooth, $W_S\hspace{0.2mm}$-finite morphism of schemes.
	\begin{enumerate}[label=(\roman*),topsep=-8pt]
		\item   For an \'etale ring homomorphism $R_0'\hspace{-1mm}\to\hspace{-1mm} R_0$  the induced map
		$\T_{W_S}(\Omega^q_{X/R_0'})\hspace{-1mm}\to \hspace{-1mm}\T_{W_S}(\Omega^q_{X/R_0})$ is an isomorphism.
		\item  If $X_{\Q}$ is smooth over $\spec (R_0)_{\Q}$ of relative dimension $d$, then $\T_{W_S}(\Omega^q_{X/R_0})=0$ for all $q>d$.
	\end{enumerate}\nointerlineskip
	\begin{proof}
		By \thref{BorKal} the morphism $W_S(R_0)\to W_S(R_0')$ is \'etale. Thus the first statement follows from \thref{etalebottom}.
		The second assertion follows from \thref{embext}. We consider the functorial commutative diagram
		\begin{mycenter}[-2mm]
			\begin{tikzcd}
				(\coprod_{n\in S} X)_{\Q}\arrow[r, "\gh_S\times \id_{\Q}","\sim"']
				&W_S(X)_{\Q}\arrow[r, "\iota_{W_S}",hook]
				& W_S(X)\\
				\coprod_{n\in S}X_{\Q}\arrow[rr, "\gh_S","\sim"']\arrow[u,"\rsim"',"\can"]
				&&W_S(X_{\Q}).\arrow[u, "W_S(\iota)"',hook]
			\end{tikzcd}
		\end{mycenter}
		 and embed $\T_{W_S}(\Omega^q_{X/R_0})$ into $(\iota_{W_S}\circ(\gh_{S}\times\id_{\Q})\circ \can)_*(\Omega^q_{X_{\Q}/R_{0,\Q}})^S=(W_S(\iota)\circ\gh_{S})_*(\Omega^q_{X_{\Q}/R_{0,\Q}})^S$,
		which vanishes for $q>d$.
	\end{proof}
\end{cor}

\section[The relative de Rham-Witt complex as a torsionless quotient]{\scshape{The de Rham-Witt complex as a torsionless quotient}}\label{sectionreldrwcomplexes}
In the previous section we have constructed a family of Witt complexes with respect to so-called generically smooth morphisms. 
If we want to find de Rham-Witt complexes, i.e. initial objects, this is not sufficient any more. 
We will need to require the morphism $X\to \spec R_0$ to be smooth (everywhere).
In this section we are going to prove
\phantomsection\label{maintheoremintro}
\begin{theoremC}\thlabel{maintheorem}\label{MAINTHEOREM}
	Let $R_0$ be a flat $\Z$-algebra, $f\colon X\to \spec R_0$ a smooth morphism of schemes. 
	If $f$ is $W_S\hspace{0.2mm}$-finite, then $\T_{W_S}(\Omega^{\bullet}_{X/R_0})$ is a $W_S(f)^{-1}\OO_{\spec W_S( R_0)}$-dga. In degree $q\geq 0$ it is the unique coherent $W_S(\OO)$-module $\T_{W_S}(\Omega^{q}_{X/R_0})$ for the \'etale topology of $W_S(X)$ such that $$\T_{W_S}(\Omega^{q}_{X/R_0})(W_S(U))=W_S\Omega^q_{R/R_0}$$
	holds for every \'etale map $U=\spec R\to  X$.
	If moreover $f$ is $W$-finite, then the functor
	\begin{equation*}
		(\text{finite truncation sets})\to (W(R_0)\text{-dga's}),~
		S\mapsto  T_{W_{S}}(\Omega_{R/R_0}^{\bullet})
	\end{equation*}
	defines the relative de Rham-Witt complex over $R$ with $W(R_0)$-linear differential.
\end{theoremC}
We make use of the fact that smooth morphisms locally factor through an \'etale morphisms followed by the structure map 
$\A^t_{R_0}=\spec R_0[x_1,\dots,x_t]\to \spec R_0$ of an affine $t$-space over the base for some $t\geq 0$. Thus we focus on polynomial algebras first.

\begin{lemma}\thlabel{dRWtorsionless}
	Given a finite truncation set $S$, let $R_0$ be an $S$-torsion free ring and \\$R= R_0[x_1,\dots, x_t]$ a polynomial algebra for some $t\geq 0$. The
	$W_S(R)$-module $W_S\Omega_{R/R_0}^{\bullet}$ is torsionless, i.e. we have $T_{W_S(R)}(W_S\Omega_{R/R_0}^{\bullet})=W_S\Omega_{R/R_0}^{\bullet}$.
	\begin{proof}
		Since the torsionless quotient commutes with direct sums, we can check the statement degreewise for every $q\geq 0$. To simplify computations we transfer the problem to the language of \citeauthor{CD15}.\footnote{see appendix \ref{AppendixB} for a brief recall of Cuntz and Deninger's $\Z$-subalgebra $X_S(R)\subseteq R^S$ generated by elements $V_n\la r \ra\defeq (n\delta_{k|n}r^{n/k})_{k\in S}, r\in R$, the associated dga $X_S^{\bullet}(R)$ and the relative de Rham-Witt complex $E_S\Omega_{R/R_0}^{\bullet}$}
		In this setting we have $X_S(R_0)=\gh_S(W_S(R_0))$ and $X_S(R)=\gh_S(W_S(R))$. 
		By \cite[Lemma 3.11]{CD15}, which uses \cite[Cor. 2.18]{LZ04} as a main tool, the natural map
		$$\Omega_{X_S(R)/X_S(R_0)}^{q}\overset{\GG_S}{\to}(\Omega_{R/R_0}^{q})^S, r_0dr_1\cdots dr_q\mapsto \GG_S(r_0)d\GG_S(r_1)\cdots \GG_S(r_q)$$
		induced by the ghost maps on $X_S(R_0)$ and $X_S(R)$, gives us an injection 	$$X_S^{q}(R)\cong E_S\Omega^{q}_{R/R_0}\hookrightarrow(\Omega_{R/R_0}^{q})^S$$
		of $X_S(R_0)$-dga's.
		By the functoriality of the torsionless quotient it is enough to show that $(\Omega_{R/R_0}^{q})^S$ is torsionless. Then $X_S^q(R)$ is a torsionless $X_S(R)$-module as a submodule of a torsionless module (cf. \thref{torsionlesssub}).\\
		Let $\omega=(\omega_k)_{k\in S}\in (\Omega_{R/R_0}^{q})^S$ be homogeneous of degree $q$ such that $\ev_{\omega}=0$. Then er have $\varphi(\omega)=0$ for all $X_S(R)$-linear maps $\varphi\colon (\Omega_{R/R_0}^{q})^{S}\to X_S(R)$. 
		Since $R$ is a polynomial algebra over $R_0$ and $d$ is an $R_0$-linear derivation, we can write 
		 $$\omega_m=\sum_{1\leq j_1<\dots<j_q\leq t}f_{m,(j_1,\dots,j_q)}dx_{j_1}\cdots dx_{j_q}$$ 
		with uniquely determined polynomials $f_{m,(j_1,\dots ,j_q)}\in R=R_0[x_1,\dots,x_r]$. 
		We claim that the maps 
		$$\varphi_{m,(i_1,\dots,i_q)}\colon \omega\mapsto \sigma^{q+1}\left(\delta_{k=m} f_{m,(i_1,\dots,i_q)}\right)_{k\in S}$$
		with  $\sigma\defeq \lcm(S)$ lie in $\Hom_{X_S(R)}((\Omega_{R/R_0}^{q})^S,X_S(R))$ for all $m\in S, 1\leq i_1<\dots<i_q\leq t$. 
		First of all we prove that the image is contained in $X_S(R)$ starting with $m=1$. Consider the element
		 \begin{equation*}
			a^{(1)}\defeq \sum_{k\in S}\mu(k)\frac{\sigma^{q+1}}{k}V_k\la f_{1,(i_1,\dots,i_q)}^{k}\ra
			=\sigma^{q+1}\sum_{k\in S}\mu(k)\left(\delta_{k|l}f_{1,(i_1,\dots,i_q)}^{l} \right)_{l\in S}\in X_S(R),
		\end{equation*}
		where $\mu\colon \N\to \{-1,0,1\}$ is the M\"obius function.
		Its first coefficient is of the form\\
		$
		a_1^{(1)}=\sigma^{q+1}(-1)^0f_{1,(i_1,\dots,i_q)}=\varphi_{1,(i_1,\dots,i_q)}(\omega)_1$. For $l>1$ the M\"obius function satisfies $\sum_{k|l}\mu(k)=0$. We conclude
		 \begin{equation*}
			a_l^{(1)}
			=\sigma^{q+1}\sum_{\substack{k|l}}\mu(k)f_{1,(i_1,\dots,i_q)}^{l}=0=\varphi_{1,(i_1,\dots,i_q)}(\omega)_l
		\end{equation*}
		for $l> 1$.
		Now let $1< m\in S$. Then the element
		 \begin{equation*}
			a^{(m)}\defeq
			\sum_{k\in S,m|k}\mu\Big( \frac{k}{m}\Big)\frac{\sigma^{q+1}}{k}V_k\la f_{m,(i_1,\dots,i_q)}^{\frac{k}{m}}\ra
			=\sum_{k\in S,m|k}\mu\Big( \frac{k}{m}\Big)\sigma^{q+1}\left(\delta_{k|l}f_{m,(i_1,\dots,i_q)}^{\frac{l}{m}} \right)_{l\in S}\in X_S(R)
		\end{equation*}
		has coefficient
		 $a_l^{(m)}=\delta_{l=m}\sigma^{q+1}f_{m,(i_1,\dots,i_q)}$ at $l\in S$, which is equal to $\varphi_m(\omega)_l$.
		Hence we obtain
		 $$\varphi_{m,(i_1,\dots,i_q)}(\omega)=\sum_{k\in S,m|k}\mu\Big(\frac{k}{m}\Big)\frac{\sigma^{q+1}}{k}V_k\la f_{m,(i_1\dots i_q)}^{\frac{k}{m}}\ra\in X_S(R).$$
		The $X_S(R)$-linearity is obvious.\\
		Now let $\varphi_{m,(i_1,\dots,i_q)}(\omega)=0$ for all $m\in S$ and $1\leq i_1<\dots<i_q\leq t$. Since $X_S(R)$ has no $\Z$-torsion, we can divide by the power of $\sigma$ and conclude $f_{m,(i_1,\dots,i_q)}=0$ for all $i_1<\dots<i_q$ in each coefficient $\omega_m$. Hence $\omega=0$ and $(\Omega_{R/R_0}^q)^S$ is torsionless as an $X_S(R)$-module.
	\end{proof}
\end{lemma}
\begin{proof}[Proof of \thref{maintheorem}.]
	By assumption, the scheme $W_S(X)$ is locally noetherian and the induced morphism on Witt schemes $W_S(f)\colon W_S(X)\to \spec  W_S(R_0)$ is locally of finite type. Moreover, the morphism
	$W_S(X)_{\Q}\to\spec W_S(R_0)_{\Q} $ is smooth because it corresponds to the disjoint union of $|S|$ copies of the smooth homomorphism $X_{\Q}\to \spec R_{0,\Q}$ via the ghost maps on $X$ and $R_0$. By \thref{subschemenoetherian} the scheme $X$ and therefore also $X_{\Q}$ is locally noetherian. Thus the map $W_S(X)\to \spec W_S(R_0)$ is generically smooth. Besides, the $W_S(\OO_X)$-modules $\Omega^q_{W_S(X)/W_S(R_0)}$, its torsionless quotient and $W_S\Omega_{ X/R_0}^q$ are coherent by \cite[ \href{https://stacks.math.columbia.edu/tag/01V2}{01V2}, \href{https://stacks.math.columbia.edu/tag/01XZ}{01XZ (2)}, \href{https://stacks.math.columbia.edu/tag/01Y1}{01Y1}]{StacksProject}. 
	Therefore the first part of the statement follows from \thref{dga} and \thref{etalesheaf}.\\
	For every $x\in X$ there exists an affine open neighbourhood $U=\spec R\subseteq X$ and a factorization $f|_U\colon U\overset{g}{\to} \Af_{R_0}^t\to \spec R_0$,
	where $g$ is \'etale and $t$ is the relative dimension of $f$ at $x$ by \cite[ \href{https://stacks.math.columbia.edu/tag/039Q}{039Q}]{StacksProject}.
	Applying the Witt vector functor $W_S$ yields the factorization $W_S(f|_U)\colon W_S(U)\overset{W_S(g)}{\to} W_S(\Af_{R_0}^t)\to \spec W_S(R_0)$. Here $W_S(g)$ is \'etale by \thref{BorKal}. 
	We denote the polynomial ring in $t\geq 0$ variables over $R_0$ by $R'$ and use \thref{etalechat} to compute
	$W_S(g)^*W_S\Omega_{\Af_{R_0}^t/R_0}^q
		\cong(W_S\Omega_{R'/R_0}^q\otimes_{W_S(R')}W_S(R))\widetilde{~~~}
		\cong W_S\Omega_{U/R_0}^q$.
	By \thref{flatpullbacktorsionless} (iv) and \thref{dRWtorsionless} the sheaf $W_S\Omega_{U/R_0}^q$ is torsionless. As we can cover $W_S(X)$ by the open subschemes $W_S(U)$ (see \cite[Theorem A]{Bor15}), we conclude that the sheaf $W_S\Omega_{ X/R_0}^q$ is torsionless by \thref{cohsurjflat}.
	The affine scheme $W_S(U)$ is noetherian as an open subscheme of the locally noetherian scheme $W_S(X)$  and $\Omega_{W_S(U)/W_S(R)}^q$ is a coherent $\OO_{W_S(U)}$-module (see \cites[Prop. 16.5]{Bor15}[ \href{https://stacks.math.columbia.edu/tag/01OW}{01OW}]{StacksProject}).
	By \thref{cohsurjflat} the global sections of
	$\T_{W_S}(\Omega_{U/R_0}^q)$ are given by the $W_S(R)$-module $T_{W_{S}}(\Omega_{R/R_0}^{q})$.\\
	By \thref{surjreldrw}, i.e. the universal $\text{property of the}$ K\"ah\-ler differential forms, there are unique ho\-mo\-mor\-phisms of $W_S(R_0)$-dga's 
	\begin{equation*}
		\alpha\colon \Omega_{W_S(R)/W_S(R_0)}^{\bullet}\to W_S\Omega_{R/R_0}^{\bullet}\quad\text{and} \quad
		\beta\colon \Omega_{W_S(R)/W_S(R_0)}^{\bullet}\to T_{W_S}(\Omega_{R/R_0}^{\bullet})
	\end{equation*}
	functorial in $S$, which are the identity in degree zero.\\
	By \thref{wittcomplex} the torsionless quotient of the de Rham complex is a Witt complex and $W_S\Omega^{\bullet}_{R/R_0}$ is torsionless as we have already explained.
	We use the functoriality of the torsionless quotient to obtain the morphism
	\begin{equation*}
		T(\alpha)\colon T_{W_{S}}(\Omega_{R/R_0}^{\bullet})\to T_{W_{S}}(W_S\Omega_{R/R_0}^{\bullet})\cong  W_S\Omega_{R/R_0}^{\bullet}
	\end{equation*}
	of Witt complexes.
	By the universal property of the de Rham-Witt complex we have a unique morphism of Witt complexes
	$$f\colon W_S\Omega_{R/R_0}^{\bullet}\to T_{W_{S}}(\Omega_{R/R_0}^{\bullet}),$$
	which fits into the commutative diagram
	 \begin{equation*}
		\begin{tikzcd}
			&\Omega_{W_S(R)/W_S(R_0)}^{\bullet}\arrow[ld, two heads,"\beta"']\arrow[d, two heads,"\alpha"]\arrow[rd, two heads,"\beta"]&\\
			T_{W_{S}}(\Omega_{R/R_0}^{\bullet})\arrow[r,"T(\alpha)"]
			&W_S\Omega_{R/R_0}^{\bullet}\arrow[r,"f"]&T_{W_{S}}(\Omega_{R/R_0}^{\bullet})
		\end{tikzcd}
	\end{equation*}
	of $W_S(R_0)$-dga's.
	All of these maps are the identity on $W_S(R)$ in degree zero and all objects are generated by elements of degree zero.\\
	Again by the universal property of $W_S\Omega_{R/R_0}^{\bullet}$ the composition $T(\alpha)\circ f$ is the identity on $W_S\Omega_{R/R_0}^{\bullet}$. We deduce the opposite identity $f\circ T(\alpha)=\id $ from the surjectivity of $\beta$ and the relation
	$\beta=f\circ \alpha=f\circ T(\alpha)\circ\beta$.
\end{proof}

\begin{remark}
	Let $ X$ be a locally noetherian, smooth $R_0$-scheme. For $S=\{1\}$ we have $$\T_{W_{\{1\}}}(\Omega_{ X/R_0}^q)=\T(\Omega_{ X/R_0}^q)=\Omega_{ X/R_0}^q$$ because $\Omega_{ X/R_0}^q$ is a locally free $\OO_X$-module.
\end{remark}

\begin{cor} Let $S$ be a finite truncation set, $R_0$ a flat $\Z$-algebra and $ X \to \spec R_0$ a $W_S\hspace{0.2mm}$-finite morphism of schemes. If $X\to \spec R_0$ is smooth of dimension $d$, then we have $\T_{W_S}(\Omega^q_{X/R_0})=0$ for all $q>d$.
	\begin{proof}
		The base change $X_{\Q}\to \spec (R_0)_{\Q}$ is also smooth of relative dimension $ d$ by \cite[ \href{https://stacks.math.columbia.edu/tag/02NK}{02NK}]{StacksProject}. Hence the assertion follows from \thref{etalebottomsmooth}.
	\end{proof}
\end{cor}

	In the setting of \thref{maintheorem} the torsionless quotient  $T_{W_S}(\Omega_{R/R_0}^{\bullet})$ is identified with its image inside $(\Omega_{ R/R_0}^{\bullet})^S$ under the naive extension of the ghost map \\$\gh_S\colon W_S(R)\to R^S$ to $\Omega_{W_S(R)/W_S(R_0)}^{\bullet}$.
	Note that for every $r\in W_S(R)$ the differential form $d\gh_n(r)=n \sum_{e|n} dr_e^{n/e-1}$ is a multiple of $n\in S$. 
	The setting of the theorem allows us to uniquely divide by $(n^q)_{n\in S}$ in degree $q\geq 0$. 
	Hence the modified ghost map 
	\begin{align*}
		\gh_S^q\colon T_{W_S}(\Omega_{R/R_0}^{q}) \quad \quad \quad \quad \quad  \to &~(\Omega_{ R/R_0}^{q})^S\\
		r_0dr_1\cdots dr_q+\ker \epsilon_{W\Omega_{S}}\mapsto &~(n^{-q})_{n\in S}\gh_S(r_0)d\gh_S(r_1)\cdots d\gh_S(r_q)\\
		&~=\gh_S(r_0)\dd \gh_S(r_1)\cdots \dd\gh_S(r_q)
	\end{align*}
	identifies the torsionless quotient with $X_S^q(R)$ in \cite[sec. 3]{CD15}.\footnote{see appendix \ref{AppendixB} for a brief recall of Cuntz and Deninger's relative de Rham-Witt complex}
\subsection[Reduction modulo an ideal]{\textbf{Reduction modulo an ideal.}} \label{Reduction modulo an ideal} If a ring homomorphism $R_0\to R$ admits a smooth, $W$-finite lifting $A_0\to A$ of flat $\Z$-algebras, then the de Rham-Witt complex can be computed via torsionless differential forms as follows:
\begin{lemma}\thlabel{reduction}
	Let $f\colon A_0 \to A$ be a smooth, $W$-finite homomorphism of flat $\Z$-algebras. Given ideals $I_0\subseteq A_0,I\subseteq A$ with $f(I_0)\subseteq I$, we set $R_0= A_0/I_0$ and $R=A/I$.\\ The functor
	\begin{align*}
		(	\text{finite truncation sets})&\to (W(R_0)\text{-dga's})\\
		S&\mapsto  T_{W_{S}}(\Omega_{A/A_0}^{\bullet})/
		T_{W_{S}(A)}(W_S(I)^{\bullet})
	\end{align*}
	defines the relative de Rham-Witt complex over $R$ with $W(R_0)$-linear differential.\\
	Here $W_S(I)^{\bullet}\subseteq \Omega_{W_{S}(A)/W_{S}(A_0)}^{\bullet}$ is the differential graded ideal generated by $W_S(I)$.
\end{lemma}
\begin{proof}
	From the commutativity of the diagram
	\begin{equation*}
		\begin{tikzcd}
			T_{W_{S}(A)}(W_S(I)^{q})\arrow[r]\arrow[d,hook]
			&T_{W_{S}}(\Omega_{A/A_0}^{q})\arrow[d,hook]
			\\
			(I^q)^{S}_{\Q}\arrow[r,hook]
			&(\Omega_{A_{\Q}/A_{0,\Q}}^{q})^{S}
		\end{tikzcd}
	\end{equation*}
	we deduce that the top horizontal map is an inclusion. Let $T_{W_S}\Omega^q_{(A,I)}$ denote the $W_S(R)$-module 
	$T_{W_{S}}(\Omega_{A/A_0}^{q})/T_{W_{S}(A)}(W_S(I)^{q})$.
	The differential on $T_{W_S}\Omega^{\bullet}_{(A,I)}$ inherited from  $T_{W_{S}}(\Omega_{A/A_0}^{\bullet})$ is a derivation which maps $W_S(R_0)$ to zero. 
	Hence there is a unique homomorphism
	$ \Omega_{W_{S}(R)/W_{S}(R_0)}^{q}\to T_{W_S}\Omega^q_{(A,I)}$
	of $W_S(R)$-modules fitting into the commutative diagram
	\begin{equation*}
		\begin{tikzcd}
			0\arrow[r]&W_S(I)^{q}\arrow[r]\arrow[d,two heads]	&\Omega_{W_{S}(A)/W_{S}(A_0)}^{q}\arrow[r]\arrow[d,two heads]&\Omega_{W_{S}(R)/W_{S}(R_0)}^{q}\arrow[d,two heads]\arrow[r]&0\\
			0\arrow[r]&T_{W_{S}(A)}(W_S(I)^{q})\arrow[r]&
			T_{W_{S}}(\Omega_{A/A_0}^{q})\arrow[r]
			&T_{W_S}\Omega^q_{(A,I)}\arrow[r]&0
		\end{tikzcd}
	\end{equation*}
	of short exact sequences. In particular, we have the equality  $\ker\epsilon_{W_S(I)^{q}}=\ker\epsilon_{\Omega W_{S}}\cap W_S(I)^q,$
	 hence the isomorphism
	\begin{align*}
		T_{W_S}\Omega^q_{(A,I)}&=T_{W_{S}}(\Omega_{A/A_0}^{q})/(W_S(I)^q/\ker\epsilon_{W_S(I)^{q}})\cong \Omega_{W_{S}(A)/W_{S}(A_0)}^{q}/(W_S(I)^q+\ker\epsilon_{\Omega W_{S}})\\
		&\cong \Omega_{W_{S}(R)/W_{S}(R_0)}^{q}/\overline{\ker\epsilon_{\Omega W_{S}}},
	\end{align*}
	where $\overline{\ker\epsilon_{\Omega W_{S}}}$ denotes the image of the submodule ${\ker\epsilon_{\Omega W_{S}}}\subseteq 
	\Omega_{ W_{S}(A)/W_S(A_0)}^q$ under the projection $	\Omega_{ W_{S}(A)/W_S(A_0)}^q\to 	\Omega_{ W_{S}(R)/W_S(R_0)}^q$.\\
	The restriction, Frobenius and Verschiebung maps from the Witt complex $T_{W_{S}}(\Omega_{A/A_0}^{q})$ all descend to $ T_{W_S}\Omega^q_{(A,I)}$ in a compatible way. Hence $S\mapsto\varprojlim_{S'\subseteq S}T_{W_{S'}}\Omega^q_{(A,I)}$ is a Witt complex over $R$ with $W(R_0)$-linear differential, where the limit runs over the directed poset of finite truncation subsets $S'\subseteq S$.
	By the universal property of the de Rham-Witt complex we obtain  unique maps of Witt complexes
	$$f'\colon W_S\Omega^{\bullet}_{R/R_0}
	\to T_{W_S}\Omega^{\bullet}_{(A,I)}\quad\text{and}\quad
	W_S\Omega^{\bullet}_{A/A_0}\to W_S\Omega^{\bullet}_{R/R_0}$$
	over $R$ and $A$, respectively.
	Since $\overline{W_S(I)^{\bullet}}\defeq T(\alpha)(T_{W_S(A)}(W_S(I)^{\bullet}))$ lies in the kernel of the second map, we obtain an induced morphism
	$$T(\alpha)'\colon T_{W_S}\Omega^{\bullet}_{(A,I)}
	\underset{\sim}{\overset{T(\alpha)}{\to}} W_S\Omega^{\bullet}_{A/A_0}/\overline{W_S(I)^{\bullet}}\to W_S\Omega^{\bullet}_{R/R_0}$$ 
	of Witt complexes over $R$ with $W(R_0)$-linear differential, which is the identity in degree zero, by \thref{maintheorem}.\footnote{For the application of \thref{maintheorem} we need the smoothness assumption on $A_0\to A$.} One can show that both compositions $f'\circ T(\alpha)'$ and $T(\alpha)'\circ f'$ are the identity by similar arguments as in the proof of \thref{maintheorem}.
\end{proof}

The proof allows us to rewrite the de Rham-Witt complex as follows:
For the image $\overline{\ker\epsilon_{\Omega W_{S}}}$ of ${\ker\epsilon_{\Omega W_{S}}}\subseteq 
\Omega_{ W_{S}(A)/W_S(A_0)}^{\bullet}$ under the projection $	\Omega_{ W_{S}(A)/W_S(A_0)}^{\bullet}\to 	\Omega_{ W_{S}(R)/W_S(R_0)}^{\bullet}$ we have the following isomorphisms of $W_S(R_0)$-dga's:
\begin{equation*}
	\Omega_{W_{S}(R)/W_{S}(R_0)}^{{\bullet}}/\overline{\ker\epsilon_{\Omega W_{S}}}\overset{\sim}{\to} 
	T_{W_{S}}(\Omega_{A/A_0}^{\bullet})/
	T_{W_{S}(A)}(W_S(I)^{\bullet})\overset{\sim}{\to} 
	W_S\Omega_{ R/R_0}^{\bullet}.
\end{equation*}

\begin{remark}\thlabel{Zcircfinite}
	In this section, the $W$-finiteness condition in the sense of \thref{Wfinite} has played an important role. From our discussion of Witt schemes in section \ref{section Witt vector rings and Witt schemes} we deduce that any morphism $ X\to \spec R_0$ locally of finite type, where $R_0$ is a finitely generated $\Z_{\circ}$-algebra, is $W$-finite.\footnote{Recall that $\Z_{\circ}$ is either equal to $\Z$, to $\Z_{(p)}$ or to $\Z_p$ for a prime number $p$ (see page \pageref{Zcirc}).} This follows from \cite[ \href{https://stacks.math.columbia.edu/tag/01T3}{01T3}, \href{https://stacks.math.columbia.edu/tag/01T6}{01T6}]{StacksProject}) and \thref{wittschemenl} (v).
	Moreover, every finitely gene\-rated $\Z_{\circ}$-algebra $R$ is noetherian and -- in this case -- smooth $R$-algebras (or $R$-schemes) are (locally) noetherian.
\end{remark}
We obtain the following specifications of \thref{maintheorem} and \thref{wittcomplex} (and of
\thref{reduction} analogously).
\begin{theorem}
	Let $R_0$ be a finitely generated, flat $\Z_{\circ}$-algebra.
	\begin{enumerate}[(i)]
		\item For a generically smooth homomorphism $R_0\to R$ the functor
		\begin{equation*}
			(\text{finite truncation sets})\to (W(R_0)\text{-dga's}),~
			S\mapsto T_{W_{S}}(\Omega_{R/R_0}^{\bullet})
		\end{equation*}
		defines a Witt complex over $R$ with $W(R_0)$-linear differential such that  the induced map
		\begin{equation*}
			T_{W_S}(\Omega^q_{R/R_0})\otimes_{W_S(R)}W_S(R') \to T_{W_S}(\Omega^q_{R'/R_0})
		\end{equation*}
		 is an isomorphism for any \'etale homomorphism $R\to R'$.
		\item \thlabel{maintheoremcirc}
		For a smooth morphism $f\colon X\to \spec R_0$ the sheaf $\T_{W_S}(\Omega^{\bullet}_{X/R_0})$ is a $W_S(f)^{-1}\OO_{\spec W_S(R_0)}$-dga. In degree $q\geq 0$ it is the unique coherent $W_S(\OO)$-module $\T_{W_S}(\Omega^{q}_{X/R_0})$ for the \'etale topology of $W_S(X)$ such that $\T_{W_S}(\Omega^{q}_{X/R_0})(W_S(U))=W_S\Omega^q_{R/R_0}$
		holds for every \'etale map $U=\spec R\to  X$.
		In particular, the functor
		\begin{equation*}
			(\text{finite truncation sets})\to (W(R_0)\text{-dga's}),~
			S\mapsto  T_{W_{S}}(\Omega_{R/R_0}^{\bullet})
		\end{equation*}
		defines the relative de Rham-Witt complex over $R$ with $W(R_0)$-linear differential.
	\end{enumerate} 
\end{theorem}
Instead of associating a dga $E_S^{\bullet}$ with every truncation set ${S\subseteq \N}$, we can fix a truncation set $S$ and associate dga's $E_{S'}^{\bullet}$ only with the truncation subsets $S'\subseteq S$.
Indeed, \thref{dga} and all statements of sections \ref{sectionwittcomplexes} and \ref{sectionreldrwcomplexes} hold true for so-called $S$-generically smooth morphisms, i.e. morphisms $X\to Y$ locally of finite type of $S$-torsion free schemes such that $X$ is locally noetherian and the base change $X_{\Q_S}\to Y_{\Q_S}$ is smooth (cf. \thref{gensmoothdef}).
The $p$-typical de Rham-Witt complex is the prototypical example. In this case it is enough to consider $p$-torsion free schemes, where multiplication by $p$ is injective on the structure sheaf.

\appendix
\addcontentsline{toc}{section}{Appendices}
\section*{\scshape{Appendices}}
\markboth{Appendices}{}
Independently of the previous analysis we give an alternative interpretation of the ghost map as a sequence of blowing ups in appendix \ref{AppendixA} and a generalization of Dwork's Lemma to the relative de Rham-Witt complex in appendix \ref{AppendixB}.
\renewcommand{\thesubsection}{A}
\subsection[A geometric interpretation of the ghost map]{\textbf{A geometric interpretation of the ghost map.}}\label{AppendixA} Let $S$ be a finite truncation set and $X$ an $S$-torsion free scheme. In this section we are going to prove the following theorem, which endows the ghost map $\gh_S\colon \coprod_{n\in S} X\to W_S(X)$ with a geometric meaning:
\begin{theorems}\thlabel{blowupS} Let $S$ be a finite truncation set and $X$ an an $S$-torsion free scheme. There is an explicit finite sequence of blowing ups 
	$$X_1\defeq\coprod_{n\in S} X \to \dots\to X_{|S|}\defeq X_S\defeq W_S(X)$$ 	
	corresponding to the ghost map $\gh_S\colon \coprod_{n\in S}X\to W_S(X)$.
\end{theorems}
Indeed, we need exactly $|S|-1$ blowing ups of the following type:
\begin{lemmas}\thlabel{onestepblowup}
	Let $p$ be a prime number, $S$ a finite truncation set and $X$ an $S$-torsion free scheme.
	The blowing up of $X_S=W_S(X)$ along the closed subscheme $X_p\defeq W_{S(p)}(X)\times_{\spec \Z} \spec \F_p$  is given by
	\begin{equation*}
		\pi_{S,p}\defeq \iota_{S(p),S}\times F_p\colon	X_{S(p)}\sqcup X_{S/p}\to X_S.
	\end{equation*}
\end{lemmas}
If moreover $X$ is normal, we will show that the ghost map $\gh_S$ coincides with the normali\-zation of $W_S(X)$.
\begin{lemmas}\thlabel{quotwitt}
	Let $S$ be a finite truncation set and $R$ an $S$-torsion free integral domain. There are exactly $|S|$ minimal prime ideals in $W_S(R)$ given by $\p_n=\ker(\gh_n), n\in S$. Via the inclusion $\gh_S\colon W_S(R)\hookrightarrow R^S$ the total ring of fractions of $W_S(R)$ is given by $Q(W_S(R))=Q(R)^S$. 
	\begin{proof}
		The statement is trivial for $S=\{1\}$, so we assume $|S|\geq 2$.
		As $R$ is an integral domain, the minimal prime ideals in $W_S(R)$ are of the form $\p_n\defeq \gh_S^{-1}(\q_n)$ for some $n\in S$ where $\q_n=(0)\times A^{S\setminus\{n\}}$ with $(0)$ at the $n$-th position, i.e. given by the kernels of the maps $\gh_n=\pr_n\circ \gh_S$. In particular, there are at most $|S|$ of them. 
		Given a prime number $p\in S$ we have $V_p[1]\notin \p_m$ for all $m\in p\cdot S/p$. Hence $\p_1\neq \p_m$ for all $1\neq m\in S$. 
		For $1\neq n\in S$ and $\sigma\defeq \lcm (S)$ the $k$-th ghost component of $$\mu_n\defeq \sum_{m\in n\cdot S/n}\mu\left(\frac{m}{n}\right)\frac{\sigma}{m}~ V_m[1]\in W_S(R),$$
		where $\mu \colon \N\to \{-1,0,1\}$ is the M\"obius function, is equal to $\sigma$ for $n=k$ and zero otherwise. 
		Thus there are exactly $|S|$ minimal prime ideals.\\
		As $R$ is an integral domain, $W_S(R)\subseteq R^S$ is reduced and its set of zero-divisors is given by the union of its minimal prime ideals (see \cite[ \href{https://stacks.math.columbia.edu/tag/00EW}{00EW (3)}]{StacksProject}).
		Thus by \cite[ \href{https://stacks.math.columbia.edu/tag/02LX}{02LX}]{StacksProject} the total ring of fractions is given by
		$Q(W_S(R))=\prod_{n \in S}W_S(R)_{\p_n}$. 
		As soon as we have shown $W_S(R)_{\p_n}\cong Q(R)$ for every $n\in S$ we are done.
		Consider the commutative diagram
		\begin{mycenter}[-2mm]
			\begin{tikzcd}
				W_S(R)_{\p_n}\arrow[r,"\gh_S"]\arrow[dr,"\gh_S"']
				&(R^S)_{\q_n}\arrow[d]\\
				&(R^S)_{\gh_S(\p_n)},
			\end{tikzcd}
		\end{mycenter}
		where the perpendicular map is induced by the inclusion $\gh_S(\p_n)\subseteq \q_n$ and the ring $(R^S)_{\gh_S(\p_n)}$ is the localization of $R^S$ at the multiplicative subset $R^S\setminus \gh_S(\p_n)$. 
		The dia\-gonal homomorphism is injective by exactness of localization. Hence the horizontal map is injective. 
		The projection $\pr_n\colon R^S\to R$ induces an isomorphism $\pr_n\colon (R^S)_{\q_n}\to Q(R)$. 
		Observe that $\gh_S(V_n[r]/V_n[s])$ is a representative of the preimage of $r/s\in Q(R)$ in $(R^S)_{\q_n}$ under $\pr_n$ to conclude that the horizontal map is an isomorphism.
	\end{proof}
\end{lemmas}
The normalization of the ring of Witt vectors over a normal domain is simply given by its ghost map (cf. \cite[Cor. 8.2]{Bor15a}).
\begin{cors}\thlabel{normalscheme}
	Let $S$ be a finite truncation set and $X$ an $S$-torsion free, normal scheme such that every quasi-compact open subscheme of $W_S(X)$ has finitely many irreducible components. The normalization of $W_S(X)$ is given by the ghost map
	$\gh_S\colon \coprod_{n\in S} X\to W_S(X)$.
	\begin{proof}
		By \cite[ \href{https://stacks.math.columbia.edu/tag/035P}{035P}]{StacksProject} it is enough to compute the normalization for affine open subschemes $\spec R\subseteq X$.
		For $n\in S$ we set $R_n\defeq  R/\p_n \cong \im(\gh_n)\subseteq R$ with $\p_n=\ker(\gh_n)$. 
		To see that $R$ is integral over $R_n$ consider the monic polynomial $f(t)\defeq \sum_{k|n} k\cdot x^{n/k}\in \Z[x]$.
		For $r\in R$ the element $f(r)=\sum_{k|n} k\cdot r^{n/k}\in R_n$ is the $n$-th ghost component of the Witt vector ${(r,\dots,r)\in W_S(R)}$. As $f(x)-f(r)$ is a monic polynomial with coefficients in $R_n$, $r$ is integral over $R_n$.
		Finally, the integral closures of $R_n$ and $R$ in $Q(R)$ coincide and $R$ is integrally closed in $Q(R)$ by assumption.
		Varying over all $n\in S$ yields the claim.
	\end{proof}
\end{cors}

For an $S$-torsion free ring $R$ and a prime number $p$ we define the ideal
$$I_p\defeq (p,V_m[r]\mid m\in p\cdot S/p,~r\in R)\subseteq W_S(R),$$
the kernel of the surjection
$\pr\circ \res^{S}_{S(p)}\colon W_S(R)\twoheadrightarrow W_{S(p)}(R)/pW_{S(p)}(R)$.
\begin{proof}[Proof of \thref{onestepblowup}]
	We may assume $X_S=\spec W_S(R)$ for an affine open subscheme $\spec R\subseteq X$.
	By \cite[ \href{https://stacks.math.columbia.edu/tag/0804}{0804}]{StacksProject} the blowing up $\proj (\bigoplus_{k\geq 0}I_p^k)$ of $X_S$ in $I_p$ can be covered by the basic homogeneous open subschemes 
	$D_+(x^{(1)})=\spec( W_S(R)[I_p/x])$,
	where $x$ runs over generators of $I_p$. We will work in $X_S(R)=\gh_S(W_S(R))\subseteq R^S$ from now on (see appendix \ref{AppendixB}). 
	By abuse of notation we set $I_p=(p,V_m\la r\ra\mid m\in p\cdot S/p,~r\in R)\subseteq X_S(R)$ with $V_m\la r\ra :=\gh_S(V_m[r])=(m\delta_{k|m}r^{m/k})_{k\in S}$.
	For all $r\in R$ and $m\in p\cdot S/p$ the inclusion
	\begin{equation*}
		D_+\left(V_m\la r \ra^{(1)}\right)
		=D_+\left(\left(V_m\la r \ra^{(1)}\right)^2\right)
		=D_+\left(p^{(1)}\right)\cap D_+(m/p\cdot V_{m}\la r^2 \ra^{(1)})
		\subseteq D_+\left(p^{(1)}\right)
	\end{equation*}
	tells us that $\proj (\bigoplus_{k\geq 0}I_p^k)=D_+(p^{(1)})$ is affine (see \cite[ \href{https://stacks.math.columbia.edu/tag/00JP}{00JP (2)}]{StacksProject}). 
	It can be identified with spectrum of the homogeneous localization
	\begin{align*}
		X_S(R)\left[ \frac{I_p}{p}\right]
		=X_S(R)+\sum_{\substack{m\in p\cdot S/p\\r\in R}}\Z\cdot \frac{V_m\la r\ra}{p}
		=\sum_{\substack{m\in S(p)\\r\in R}}\Z\cdot V_m\la r\ra+\sum_{\substack{m\in p\cdot S/p\\r\in R}}\Z\cdot \frac{V_m\la r\ra}{p}
	\end{align*}
	considered as a subring of $R^S$. Indeed, the multiplication of the generators follows the rule
	\begin{align*}
		\frac{V_m\la r\ra }{\delta_p(m)}~\frac{V_n\la s\ra }{\delta_p(n)}
		=\delta_S([m,n])\frac{(m,n)}{\delta_p((m,n))}\frac{V_{[m,n]}\la r^{\frac{n}{(m,n)}}s^{\frac{m}{(m,n)}}\ra}{\delta_p([m,n])}\text{ 	with }
		\delta_p(m)\defeq \begin{cases}
			p&,~p\mid m\\
			1&,~p\nmid m
	\end{cases}\end{align*}for $m\in\N$.
	Let us view the restriction $\res^S_{S(p)}$ and the Frobenius $F_p$ as ring homomorphisms from $R^S$ to $R^{S(p)}$ and to $R^{S/p}$ on the ``ghost side'', respectively.
	We claim that the homomorphism
	\begin{align*}
		\res^S_{S(p)}\times F_p\colon 	X_S(R)\left[ \frac{I_p}{p}\right] &\to X_{S(p)}(R)\times X_{S/p}(R)\\
		(a_k)_{k\in S}&\mapsto ((a_k)_{k\in S(p)},(a_{kp})_{k\in S/p})
	\end{align*} 
	is an isomorphism. 
	It is well-defined as $\res^S_{S(p)}(V_m\la r \ra_{S/m} )$ vanishes for $m\in p\cdot S/p$ and 
	\begin{align*}
		F_p\left(\frac{V_m\la r \ra_{S/m}}{\delta_p(m)}\right)
		=\begin{cases}
			\delta_{S/p}(m)V_m\la r^p\ra_{S/mp} &,~m\in S(p)\\
			V_{m/p}\la r \ra _{S/m}&,~m\in p\cdot S/p
		\end{cases}
	\end{align*}
	lies in $X_{S/p}(R)$.
	Moreover, it is surjective because we have
	$$(\res^S_{S(p)}\times F_p)(V_m\la r \ra_{S/m}-	\delta_{S/p}(m)V_{mp}\la r^p\ra_{S/mp}/p)
	=(V_m\la r \ra_{S(p)/m},	0)$$
	for $m\in S(p)$ and 
	$(\res^S_{S(p)}\times F_p)(V_{mp}\la r \ra_{S/mp}/p)
	=(0,V_m\la r \ra_{S/mp})$
	for $m\in S/p$. 
	The injectivity is clear by the decomposition $S=S(p)\sqcup p\cdot S/p$.
	Composing the isomorphism $\res^S_{S(p)}\times F_p$ with the inclusion $X_S(R)\hookrightarrow X_S(R)\left[\frac{I_p}{p}\right]$ yields the blowing up morphism
	\begin{equation*}
		\res^S_{S(p)}\times F_p\colon 	X_S(R)\hookrightarrow X_{S(p)}(R)\times X_{S/p}(R).
	\end{equation*}
	Since the ghost map is compatible with the restriction and the Frobenius, we are done.
\end{proof}

\begin{proof}[Proof of \thref{blowupS}]
	The theorem follows by successively applying the lemma to prime numbers contained in the truncation sets $S(p)$ and $S/p$ until ending up with exactly $|S|$ copies of the truncation set $\{1\}$ corresponding to the disjoint union $\spec R^S=\coprod_{n\in S} \spec R$.\\ 	
	We are left to show that the composition of these blowing ups coincides with the ghost map. We proceed via induction on the cardinality of $S$. 
	For $S=\{1\}$ there is nothing to show. 
	For $S=\{1,p\}$ and $r=(r_1,r_p)\in W_S(R)$ we have $(\res^S_{\{1\}}(r),F_p(r))=(r_1,r_1^p+pr_p)=(\gh_1(r),\gh_p(r))$. 
	For $|S|=m+1$ and a prime number $p\in S$ the cardinalities of $S(p)$ and $S/p$ are bounded above by $m$. 
	Hence blowing them up stepwise corresponds to the respective ghost maps $\gh_{S(p)}$ and $\gh_{S/p}$ by induction hypothesis. 
	We consider the image $(\res^S_{S(p)}(r),F_p(r))$ of $r\in W_S(R)$ in $W_{S(p)}(R)\times W_{S/p}(R)$ under the blowing up at $p$. The claim follows from the observations $$(\gh_{S(p)}\circ\res^S_{S(p)})(r)=(\gh_k(r))_{k\in S(p)}\quad\text{and}\quad(\gh_{S/p}\circ F_p)(r)=(\gh_{kp}(r))_{k\in S/p}$$ and the decomposition ${S=S(p)\sqcup p\cdot  S/p}$.
\end{proof}

From \thref{normalscheme} we deduce that the normalization (corresponding to the ghost map) is a finite sequence of blowing ups.

\begin{cors}
	Let $S$ be a finite truncation set and $X$ an $S$-torsion free, normal scheme such that every quasi-compact open subscheme of $W_S(X)$ has finitely many irreducible components. The normalization $\gh_S\colon \coprod_{n\in S} X\to W_S(X)$ of $W_S(X)$ is given by the explicit finite sequence of blowing ups in \thref{blowupS}.
\end{cors}

\renewcommand{\thesubsection}{B}
\subsection[Dwork's Lemma for the relative de Rham-Witt complex]{\textbf{Dwork's Lemma for the relative de Rham-Witt complex.}}\label{AppendixB} 
\label{ChapterDworksLemma}
We are going to characterize the image of the de Rham-Witt complex under the ghost map in terms of congruences modulo prime powers. We will be working in the language of \citeauthor{CD15} (see \cite[sec. 3]{CD15}), but only consider finite truncation sets $S$ and smooth homomorphisms $f\colon R_0\to R$ of flat $\Z$-algebras. We briefly recall the setting: Let \begin{equation}\tag{$\pi$}
	\begin{tikzcd}
		0\arrow[r]
		&I_0\arrow[r]\arrow[d,hook]
		&A_0\arrow[r,"\pi_0"]\arrow[d,hook]
		&R_0\arrow[r]\arrow[d,"f"]
		&0\\
		0\arrow[r]
		&I\arrow[r]
		&A\arrow[r,"\pi"]
		&R\arrow[r]
		&0
	\end{tikzcd}
\end{equation}
be any presentation of $R_0\to R$, i.e. a commutative diagram of exact rows where $A_0$ is a flat $\Z$-algebra and $A$ a polynomial ring over $A_0$. Let $X_S(A)\subseteq A^S$ denote the subring generated by elements $V_n\langle a\rangle\defeq \gh_S(V_n[a])=(n\delta_{k|n}a^{n/k}){k\in S}$ with $a\in A, n\in S$, $X_S(I)\subseteq X_S(A)$ the ideal generated by elements $V_n\langle a\rangle$ with $a\in I, n\in S$. We set $E_S(R,\pi)\defeq X_S(A)/X_S(I)$ (and analogously for $I_0, A_0$ and $R_0$).
\\Moreover, we denote by $X_S^{\bullet}(A)\subseteq ((\Omega_{A_\Q/A_{0,\Q}}^{\bullet})^S, \dd\defeq (\frac{1}{n}d)_{n\in S})$ the differential graded subring generated by $X_S(A)$ and by $X_S^{\bullet}(I,A)\subseteq X_S^{\bullet}(A)$ the differential graded ideal generated by $X_S(I)$. We set $E_S\Omega^{\bullet}(\pi)\defeq X_S^{\bullet}(A)/X_S^{\bullet}(I,A)$.\\
The ring $E_S(R,\pi)$ and the $E_S(R_0,\pi_0)$-dga $E_S\Omega^{\bullet}(R,\pi)$ can be identified with the ring of Witt vectors and the relative de Rham-Witt complex and are equipped with natural restriction, Frobenius and Verschiebung maps (see \cite[diag. (25), Thm. 3.12]{CD15}).\\
Under the inclusion $X_S^{\bullet}(A)\hookrightarrow (\Omega_{A/A_0}^{\bullet})^S$ the components of $X_S^{\bullet}(I,A)$ are mapped to the differential graded ideal $I^{\bullet}$ inside $\Omega_{A/A_0}^{\bullet}$ generated by $I=\ker(\pi)\subseteq A$. 
The inclusion above induces a natural homomorphism
$$\GG_S^{\bullet}\colon E_S\Omega^{\bullet}(\pi)=X_S^{\bullet}(A)/X_S^{\bullet}(I,A)\to (\Omega_{A/A_0}^{\bullet}/I^{\bullet})^S \overset{\sim}{\to}(\Omega_{R/R_0}^{\bullet})^S$$
of graded rings, whose degree zero component is the ghost map ${\GG_S^{0}=\GG_S\colon E_S(R,\pi)\to R^S}$.
It is compatible with Frobenius and Verschiebung maps and satisfies $d\GG_S^{\bullet}=(m)_{m\in S}\GG_S^{\bullet}\dd $.
From now on we fix a prime $p\in S$ and assume that $R$ is equipped with a Frobenius lift $\phi_p\colon R\to R$ and $R_0$ with an endomorphism $\varphi_p\colon R_0\to R_0$ satisfying $f\varphi_p=\phi_p f$. Then $\phi_p$ induces a degree preserving endomorphism on $\Omega_{R/R_0}^{\bullet}$.\\
Note that $p^{q}$ divides $d\phi_p(\omega)$ if $\omega$ is homogeneous of degree $q-1$. By abuse of notation we define the following homomorphism of $R$-modules for $q\geq 0$:
\begin{align*}
	\phi_p\colon \Omega_{R/R_0}^{q}~~~~~~~~&\to\Omega_{R/R_0}^{q}\\
	r_0dr_1\cdots dr_q&\mapsto p^{-q}\phi_p(r_0)d\phi_p(r_1)\cdots d\phi_p(r_q)
\end{align*}
The statement of the following theorem was suggested by \citeauthor{Chat12}. We are going to work out a proof in this section.\phantomsection\label{DworkdeRhamintro}
\begin{theorems}\thlabel{DworkdeRham}
	Let $f\colon R_0\to R$ be a smooth morphism of flat $\Z$-algebras and $S\hspace{-1mm}=\hspace{-1mm}\{1,p,...,p^{n}\}$, $n\geq 0$. Assume that $R$ is equipped with a Frobenius lift $\phi_p\colon R\to R$ and $R_0$ with an endomorphism $\varphi_p\colon R_0\to R_0$ such that $f \varphi_p=\phi_p f$.
	For $\omega=(\omega_k)_{k=0}^n\in (\Omega_{R/R_0}^{\bullet} )^S$ the following are equivalent:
	\begin{enumerate}[(i)]
		\item the element $\omega$ lies in	the image of the ghost map $\GG_S^{\bullet}\colon E_S\Omega^{\bullet}(\pi)\to (\Omega_{R/R_0}^{\bullet} )^S$;
		\item the element $\omega$ satisfies $\phi_p(\omega_k)-\omega_{k+1}\in p^{k+1}\Omega_{R/R_0}^{\bullet}+d\Omega_{R/R_0}^{\bullet}$ for all $0\leq k\leq n-1$.
	\end{enumerate}
\end{theorems}

Note that -- in comparison to \hyperref[Dwork'sLemma]{Dwork's Lemma} on the level of Witt vector rings -- it is not enough to consider congruences modulo powers of $p$ in degrees $q>0$:

\begin{exams}
	Consider the ring homomorphism $\Z\to \Z[t]$ and the module $X_S^1(\Z[t])$ with respect to $S=\{1,2\}$. For $\omega=(\omega_1,\omega_2)=\dd V_2\la t \ra$ we have
	$\phi_2(\omega_1)-\omega_2
	=0-dt=-dt$.
	Obviously, this does not lie in $2\Omega^1_{\Z[t]/\Z}$, but in $d\Z[t]$.
\end{exams}
We set the convention ${\Omega_{R/R_0}^{-1}= 0}$ and skip all proofs in degree $q=0$ (cf. \thref{Dwork}).

\begin{lemmas}\thlabel{Frob-1} 
	For all $r\in R$ and $m\geq 0$ we have $
		\phi_p(r^{p^m-1}dr)-r^{p^{m+1}-1}dr\in dR$.
	\begin{proof}
		We have $\ord_p\left(\binom{p^m}{i}\right)=m-\ord_p(i)$ for all $1\leq i\leq p^m$. We make use of the  factorization		$\binom{p^m-1}{i-1}=\frac{i}{p^m}\binom{p^m}{i}$ and the expansion formula
		$\binom{p^m}{i}=\binom{p^m-1}{i}+\binom{p^m-1}{i-1}$ for all $1\leq i\leq p^m-1$. For $r\in R$ there exists $s\in R$ such that $\phi_p(r)=r^p+ps$. The statement follows from the computation
		\begin{align*}
			&\phi_p(r^{p^m-1}dr)-r^{p^{m+1}-1}dr=(r^p+ps)^{p^m-1}(r^{p-1}dr+ds)-r^{p^{m+1}-1}dr\\
			=&\sum_{i=1}^{p^m-1}\binom{p^m-1}{i}p^ir^{(p^m-i)p-1}s^idr
			+\sum_{i=1}^{p^m}\binom{p^m-1}{i-1}p^{i-1}r^{(p^m-i)p}s^{i-1}ds\\
			=&\sum_{i=1}^{p^m-1}\frac{\binom{p^m}{i}}{p^{m-\ord_p(i)}}p^{i-\ord_p(i)-1} dr^{(p^m-i)p}s^i
		\end{align*}
		in $\Omega^1_{R/R_0}$ and the inequality $i-1 \geq \ord_p(i)$ for all $i\in\N$.
	\end{proof}
\end{lemmas}

We introduce the following notations for $q\geq 0$, $\overline{m}=(m_{j})_{j=0}^q\in \Z_{\geq 0}^{q+1}$ and $\overline{r}=(r_{j})_{j=0}^q\in R^{q+1}$:
\begin{align*}
	\omega_{\overline{r}}\defeq  r_{0}dr_{1}\cdots dr_{q}\in \Omega^q_{R/R_0},
	\omega_{\overline{r}}^{(\overline{m})}\defeq  r_{0}^{p^{m_0}}\Big(\prod_{j=1}^q r_{j}^{p^{m_j}-1}\Big) dr_{1}\cdots dr_{q}=\prod_{j=0}^q r_{j}^{p^{m_j}-1}\cdot \omega_{\overline{r}}\in \Omega^q_{R/R_0}.
\end{align*}
If there exists $m\in \Z_{\geq0}$ with $\overline{m}=m\cdot \overline{1}$, we omit the overlining in the notation.

\begin{props}\thlabel{congruences2}
	Let $\overline{m}\in\Z_{\geq 0}^{q+1}$.
	\begin{enumerate}[(i)]
		\item For $q\geq 1$ and $\overline{r}\in R^{q}$ we have 
		$\phi_p(\omega_{(1,\overline{r})}^{(\overline{m})})-\omega_{(1,\overline{r})}^{(\overline{m}+1)}\in d\Omega^{q-1}_{R/R_0}$.
		\item For $q\geq 0$ and $\overline{r}\in R^{q+1}$ we have $\phi_p(\omega_{\overline{r}}^{(\overline{m})})-\omega_{\overline{r}}^{(\overline{m}+1)}\in p^{m_0+1}\Omega^q_{R/R_0}+d\Omega^{q-1}_{R/R_0}$.
	\end{enumerate}
	\begin{proof}
		Let $q\geq 1$. By \thref{Frob-1} there exist $r_j'\in R$ with
		$\phi_p(r_j^{p^{m_j}-1}dr_j)=r_j^{p^{m_j+1}-1}dr_j+dr_j'$
		for every $0\leq j\leq q-1$. For $i\in\{0,1\}$ we set
		$r_j^{(0)}=r_j$ and $r_j^{(1)}=r_j'$ and compute
		\begin{align*}
			&\phi_p(\omega_{(1,\overline{r})}^{(\overline{m})})-\omega_{(1,\overline{r})}^{(\overline{m}+1)}
			=\phi_p\Big(\Big(\prod_{j=0}^{q-1}r_j^{p^{m_j}-1}\Big)dr_{0}\cdots dr_{q-1}\Big)
			-\Big(\prod_{j=0}^{q-1}r_j^{p^{m_j+1}-1}\Big)dr_{0}\cdots dr_{q-1}\\
			=&\sum_{i_0+\dots+i_{q-1}=1}^{q-1}\Big(\prod_{\substack{j=0,\\ i_j=0}}^{q-1}r_j^{p^{m_j+1}-1}\Big)dr_0^{(i_0)}\cdots dr_{q-1}^{(i_{q-1})},
		\end{align*}
		where $(i_0,\dots,i_{q-1})$ runs over ${\{0,1\}}^q$. We now analyse the individual summands and assume $i_0=\dots=i_l=1$ and $i_{l+1}=\dots=i_{q-1}=0$ for some $0\leq l\leq q-1$.
		Then
		\begin{align*}
			&r_{l+1}^{p^{m_{l+1}+1}-1}\cdots r_{q-1}^{p^{m_{q-1}+1}-1}dr_0'\cdots dr_l'dr_{l+1}\cdots dr_{q-1}\\
			=&dr_0'\cdots dr'_{l-1}d(r_l'r_{l+1}^{p^{m_{l+1}+1}-1})\cdots d(r_{q-2}r_{q-1}^{p^{m_{q-1}+1}-1}) dr_{q-1}
		\end{align*}
		proves (i).
		For part (ii) we rearrange
		\begin{equation*}
			\phi_p(\omega_{\overline{r}}^{(\overline{m})})-\omega_{\overline{r}}^{(\overline{m}+1)}\\
			=(\phi_p(r_{0}^{p^{m_0}})-r_{0}^{p^{m_0+1}})\phi_p(\omega_{(1,r_1,\dots,r_q)}^{(\overline{m})})
			+r_0^{p^{m_0+1}}(\phi_p(\omega_{(1,r_1,\dots,r_q)}^{(\overline{m})})-\omega_{(1,r_1,\dots,r_q)}^{(\overline{m}+1)}).
		\end{equation*}
		Since $\phi_p(r_{0}^{p^{m_0}})-r_{0}^{p^{m_0+1}}$ lies in $p^{{m_0}+1}R$, the first summand lies in $p^{{m_0}+1}\Omega^q_{R/R_0}$. 
		For the second summand note that there exists $\eta\in \Omega^{q-1}_{R/R_0}$  with $\phi_p(\omega_{(1,r_1,\dots,r_q)}^{(\overline{m})})-\omega_{(1,r_1,\dots,r_q)}^{(\overline{m}+1)}=d\eta$ by part (i). 
		The statement follows from the observation
		$
			r_0^{p^{m_0+1}}d\eta
			=d(r_0^{p^{{m_0}+1}}\eta)-p^{{m_0}+1}r_0^{p^{m_0+1}-1}dr_0\cdot \eta$.
	\end{proof}
\end{props}

Now the implication ``$(i) \Rightarrow (ii)$'' of \thref{DworkdeRham} is a straightforward computation.
For the opposite direction we introduce the notation
\begin{equation*}
	V_{p^m}(\overline{a};\overline{b})\defeq \sum_{i=1}^t\left(V_{p^m}\la a_{i,0}\ra \dd V_{p^m}\la a_{i,1}\ra \cdots \dd V_{p^m}\la a_{i,q}\ra  
	+\dd V_{p^m}\la b_{i,1}\ra \cdots \dd V_{p^m}\la b_{i,q}\ra\right)\in X_S^q(A)
\end{equation*}
for $0\leq m\leq n$, $t\in\N$, $\overline{a}=(a_{i,j})_{0\leq j\leq q}^{1\leq i\leq t}\in A^{ (q+1)\times t}$ and $\overline{b}=(b_{i,j})_{1\leq j\leq q}^{1\leq i\leq t}\in A^{q\times t}$.

\begin{proof}[Proof of \thref{DworkdeRham}]
 For the direction ``(ii) $\to$ (i)'' let $\omega=(\omega_k)_{0\leq k\leq n}\in (\Omega^q_{R/R_0})^{n+1}$ be homogeneous of degree $q\geq 1$ satisfying 
	$\phi_p(\omega_k)-\omega_{k+1}\in p^{k+1}\Omega^q_{R/R_0}+d\Omega^{q-1}_{R/R_0}\text{ for all }0\leq k\leq n-1$.
	We claim that there exist $t\in \N$ and $\overline{r_{0,i}},\dots, \overline{r_{n,i}}\in R^{q+1}, \overline{s_{1,i}},\dots,\overline{s_{n,i}}\in R^{q}$ for all $1\leq i\leq t$ such that 
	$$\omega_k=\sum_{i=1}^t\left(\omega_{\overline{r_{0,i}}}^{(k)}
	+\sum_{m=1}^k \Big(p^m\omega_{\overline{r_{m,i}}}^{(k-m)}+\omega_{(1,\overline{s_{m,i}})}^{(k-m)}\Big)\right)\text{ for all }0\leq k \leq n.$$
	Afterwards we can choose preimages $\overline{a_{m}}\in A^{(q+1)\times t}$ of $(\overline{r_{m,i}})_{i=1}^t$ for $0\leq m\leq n$ and $\overline{b_{m}}\in A^{q\times t}$ of $(\overline{s_{m,i}})_{i=1}^t$ for $1\leq m\leq n$
	and complete the proof by remarking
	$$\GG_S^{\bullet}(V_{p^m}(\overline{a_m};\overline{b_m})+X_S^q(I,A))
	=\sum_{i=1}^t\left(\delta_{m\leq k}\left(p^m \omega_{\overline{r_{m,i}}}^{(k-m)}+\omega_{(1,\overline{s_{m,i}})}^{(k-m)}\right)\right)_{k=0}^n.$$
	We prove the claim via induction on $k$.
	For $k=0$ there exist $t\geq 1$ and $\overline{r_{0,i}}\in R^{q+1}$ for all $1\leq i\leq t$ such that 
	$\omega_0=\sum_{i=1}^t\omega_{\overline{r_{0,i}}}=\sum_{i=1}^t\omega_{\overline{r_{0,i}}}^{(0)}$.
	The integer $t\geq 1$ might increase during the induction. 
	By additivity we may always reset $t=1$. \\
	Let us assume that the claim holds for $k=n-1$. 
	By assumption, we have
	\begin{equation*}
		\omega_n- \phi_p(\omega_{n-1})
		\hspace{-1mm}=\hspace{-1mm}\omega_n-\Big(\phi_p\left(\omega_{\overline{r_{0}}}^{(n-1)}\right)
		+\sum_{m=1}^{n-1}\hspace{-1mm} \left(p^m\phi_p\left(\omega_{\overline{r_{m}}}^{(n-m-1)}\right)\hspace{-1mm}+\hspace{-1mm}\phi_p\left(\omega_{(1,\overline{s_{m}})}^{(n-m-1)}\right)\hspace{-1mm}\right)\hspace{-1mm}\Big)\hspace{-1mm}\in\hspace{-1mm} p^n\Omega^q_{R/R_0}+d\Omega^{q-1}_{R/R_0}.
	\end{equation*}
	We apply \thref{congruences2} and obtain
	\begin{equation*}
	\omega_n- \omega_{\overline{r_{0}}}^{(n)}
	+\sum_{m=1}^{n-1} \left(p^m\omega_{\overline{r_{m}}}^{(n-m)}+\omega_{(1,\overline{s_{m}})}^{(n-m)}\right)\in p^n\Omega^q_{R/R_0}+d\Omega^{q-1}_{R/R_0}.
	\end{equation*}
	Thus there exist $\overline{r_{n}} \in R^{q+1}$ and $\overline{s_{n}}\in R^{q}$ with
	\begin{equation*}
		\omega_n=\omega_{\overline{r_{0}}}^{(n)}
		+\sum_{m=1}^{n-1} \Big(p^m\omega_{\overline{r_{m}}}^{(n-m)}+\omega_{(1,\overline{s_{m}})}^{(n-m)}\Big)
		+p^{n}\omega_{\overline{r_{n}}}+\omega_{(1,\overline{s_{n}})}.\qedhere
	\end{equation*}
\end{proof}

In the multiple prime case we can only prove the implication ``$(i)\Rightarrow(ii)$'' of \thref{DworkdeRham}:

\begin{lemmas}\thlabel{DworkdeRhamsubset}
	Let $f\colon R_0\to R$ be a smooth homomorphism of flat $\Z$-algebras and $S$ a finite truncation set equipped with Frobenius lifts $\phi_p\colon R\to R$ and endomorphisms $\varphi_p\colon R_0\to R_0$ such that $f\varphi_p=\phi_pf$ for all prime numbers $p\in S$. Then every element $\omega\in E_S\Omega^{\bullet}(\pi)$ satisfies
	\begin{equation*}
		\phi_p(\GG_k(\omega))-\GG_{kp}(\omega)\in p^{\ord_p(k)+1}\Omega_{R/R_0}^{\bullet}+d\Omega_{R/R_0}^{\bullet}
	\end{equation*} 
	for all prime numbers $p\in S$ and all $k\in S/p$.
\end{lemmas}

For $q\geq 0$, $\overline{m}=(m_{j})_{j=0}^q\in \N^{q+1}$ and $\overline{r}=(r_{j})_{j=0}^q\in R^{q+1}$ we set
\begin{equation*}
	\eta_{\overline{r}}^{(\overline{m})}\defeq  r_{0}^{m_0}\Big(\prod_{j=1}^q r_{j}^{m_j-1}\Big)dr_{1}\cdots dr_{q}
	=\Big(\prod_{j=0}^q r_{j}^{m_j-1}\Big) \omega_{(r_0,r_1,\dots,r_q)}\in\Omega_{R/R_0}^q.
\end{equation*}

\begin{props} Let $q\geq 1 $ and $n\in \N$.\thlabel{congruences4}
	\begin{enumerate}[(i)]
		\item 	For all $r\in R$ and $m\in\N$ we have \thlabel{Frob-12}
		$
			\phi_p(r^{m-1}dr)-r^{mp-1}dr\in p^{n}\Omega_{R/R_0}^1+dR$.
		\item For all $\overline{r}\in R^{q}$ and $\overline{m}\in \N^{q+1}$ we have $\phi_p(\eta_{(1,\overline{r})}^{(\overline{m})})-\eta_{(1,\overline{r})}^{(\overline{m}p)}\in p^n\Omega_{R/R_0}^q+d\Omega^{q-1}_{R/R_0}$.
		\item \thlabel{congruences3} For all $\overline{r}\in R^{q+1}$ and $\overline{m}\in \N^{q+1}$ we have
		$\phi_p(\eta_{\overline{r}}^{(\overline{m})})-\eta_{\overline{r}}^{(\overline{m}p)}\in p^{\ord_p(m_0)+1}\Omega^q_{R/R_0}+d\Omega^{q-1}_{R/R_0}.$
	\end{enumerate}
	
	\begin{proof}
		For part (i) let $k\in \N$ be defined by the relation $kp^{\ord_p(m)}=m$. We rearrange $r^{m-1}$ to $(r^{k})^{p^{\ord_p(m)}-1}r^{k-1}$ and let $l\in \N$ be an integer such that $kl- 1\in p^{n}\Z$.
		By \thref{Frob-1} there exists $s\in R$ such that 
		\begin{equation*}
			\phi_p(r^{m-1}dr)+p^{n}\Omega_{R/R_0}^1
			=l \phi_p((r^{k})^{p^{\ord_p(m)}-1}dr^{k})+p^{n}\Omega_{R/R_0}^1
			= r^{mp-1}dr+lds+p^{n}\Omega_{R/R_0}^1.
		\end{equation*}
		Part (ii) can be shown by following the strategy in the proof of \thref{congruences2} modulo $p^n\Omega_{R/R_0}^q$. Then part (iii) follows from $\phi_p(r_0^{m_0})-r_0^{m_0p}\in p^{\ord_p(m_0)+1}R$ and statement (ii) with ${n=\ord_p(m_0)}$.
	\end{proof}
\end{props}

\begin{proof}[Proof of \thref{DworkdeRhamsubset}]
		It is enough to prove the claim on the generators $\omega+ X_S^q(I,A)$ with 
	$\omega=V_{{m_0}}\la a_0\ra \dd V_{{m_1}}\la a_1\ra \cdots\dd V_{{m_q}}\la a_q\ra$,
	$a_0,\dots,a_q\in A$ and $m_0,\dots,m_q\in S$.
	We set $M\defeq\lcm\{m_0,\dots,m_q\}$ and compute
	$
		\eta_k\defeq \GG_k(\omega +X_S^q(I,A))
		=\delta_{M| k}m_0 \eta_{\overline{r}}^{(k/\overline{m})}
	$
	for $k\in S$, where $r\in R^{n+1}$ denotes the image of $a$ under the componentwise projection $A^{n+1}\to R^{n+1}$. Here $k/\overline{m}=(k/m_0,\dots,k/m_q)$ lies in $\N^{q+1}$, whenever $M$ divides $k$.\\
	We now fix a prime  $p$ and consider elements $k\in S/p$.
	If $M$ does not divide $kp$, then $\eta_k=\eta_{kp}$ vanishes and the congruence is trivially satisfied. 
	If $M$ divides $kp$, but does not divide $ k$, we obtain $\ord_p(M)=\ord_p(k)+1$ and $
		\phi_p(\eta_k)-\eta_{kp}=-m_0 \eta_{\overline{r}}^{(kp/\overline{m})}$. 
	We are done if $\ord_p(m_0)$ is equal to $\ord_p(M)$. Otherwise we can assume $\ord_p(m_1)=\ord_p(M)$ and choose $l\in \Z$ such that $kp/m_1\cdot l-1\in p^{\ord_p(k)+1}\Z$.
	We compute
	\begin{align*}
		&{m_0} \eta_{\overline{r}}^{(kp/\overline{m})}
		={m_0}\Big(\prod_{j=0}^q r_{j}^{kp/m_j-1}\Big)\omega_{(r_0,r_1,\dots,r_q)}\\
		=& d\Big(l{m_0}r_0^{{kp/m_0}}\Big(\prod_{j=1}^q r_{j}^{kp/m_j-1}\Big)\omega_{(r_1,\dots,r_q)}\Big)
		-l{m_0}r_1^{kp/m_1-1}d\Big(r_0^{kp/m_0}\Big(\prod_{j=2}^q r_{j}^{kp/m_j-1}\Big)\Big) \cdot \omega_{(r_1,\dots,r_q)}.
	\end{align*}
	in $\Omega_{R/R_0}^q/p^{\ord_p(k)+1}\Omega_{R/R_0}^q$. But the term
	\begin{align*}
		&{m_0}r_1^{kp/m_1-1}d\Big(r_0^{kp/m_0}\prod_{j=2}^q r_{j}^{kp/m_j-1}\Big) \cdot \omega_{(r_1,\dots,r_q)}\\
		=&kpr_0^{kp/m_0-1}\Big(\prod_{j=1}^q r_{j}^{kp/m_j-1}\Big)\omega_{(r_1,r_0,r_2,\dots,r_q)}
		+{m_0}r_0^{kp/m_0}r_1^{kp/m_1-1}d\Big(\prod_{j=2}^q r_{j}^{kp/m_j-1}\Big) \cdot \omega_{(r_1,\dots,r_q)}
	\end{align*}
	is congruent to zero in $\Omega_{R/R_0}^q/p^{\ord_p(k)+1}\Omega_{R/R_0}^q$ because the second summand vanishes by the Leibniz rule and the relation $drdr=0$ for $r\in R$.\\
	If $M|k$, we have $	\phi_p(\eta_k)-\eta_{kp}
		={m_0}\phi_p( \eta_{\overline{r}}^{(k/\overline{m})})-{m_0} \eta_{\overline{r}}^{(kp/\overline{m})}$.
	Since $\phi_p( \eta_{\overline{r}}^{(k/\overline{m})})- \eta_{\overline{r}}^{(kp/\overline{m})}$ is an element in ${p^{\ord_p(k/m_0)+1}\Omega_{R/R_0}^q+d\Omega_{R/R_0}^{q-1}}$ by part (ii) of \thref{congruences4}, multiplication with ${m_0}$ completes the proof.
\end{proof}

\printbibliography[title={\textsc{References}}]
\end{document}